\documentclass[12pt,table]{article}

\def\dfrac#1#2{\lower0.15ex\hbox{\large$\frac{#1}{#2}$}}

\usepackage{tikz}

\usepackage{diagbox}
\usepackage[table]{xcolor}

\usepackage{enumerate}
\usepackage{amsmath}
\usepackage{amssymb}
\usepackage{amsthm}

\newtheorem{theorem}{Theorem}[section]
\newtheorem{lemma}[theorem]{Lemma}
\newtheorem{conjecture}[theorem]{Conjecture}

\newtheorem{corollary}[theorem]{Corollary}

\theoremstyle{definition}

\usepackage{mathtools}

\def\Det{\textup{det}}

\def\pairing{\Omega}  
\def\opairing{\Gamma} 

\def\dfrac#1#2{\lower0.15ex\hbox{\large$\frac{#1}{#2}$}}

\usepackage{hyperref}  
\usepackage[bb=dsserif]{mathalpha}
\usepackage{bm}

\def\coloneq{\coloneqq} 

\allowdisplaybreaks

\def\Exp{\mathbf{E}}
\def\Pr{\mathbf{P}}

\usepackage[english]{babel}

\title{Jaeger-type orientations of random regular graphs}

\author{Catherine Greenhill\thanks{Supported by Australian Research Council Discovery Project DP250101611.}\\
\small School of Mathematics and Statistics\\[-0.5ex]
\small UNSW Sydney\\[-0.5ex]
\small NSW 2052, Australia\\
\small c.greenhill@unsw.edu.au
\and Mikhail Isaev${}^*$ \\
\small School of Mathematics and Statistics\\[-0.5ex]
\small UNSW Sydney\\[-0.5ex]
\small NSW 2052, Australia\\
\small m.isaev@unsw.edu.au
\and Charles Lewis\\
\small UNSW Sydney\\[-0.5ex]
\small NSW 2052, Australia\\
\small charles.lewis@student.unsw.edu.au
}

\date{24 April 2026}

\begin{document}

\maketitle

\begin{abstract}
We consider $p$-orientations, which are defined to be orientations of $d$-regular graphs such that every vertex either has in-degree $p$ or out-degree $p$.  These generalise the orientations considered in Jaeger's conjecture, where $d=4p+1$. Working with random $d$-regular graphs using the small subgraph conditioning method, we prove that a $d$-regular graph has a $p$-orientation with high probability for several values of $(d,p)$, including the $p=3,4$ cases of
Jaeger's conjecture (known to be deterministically false). Some negative results are obtained by 
exploiting a connection with maximum bisection size.
\end{abstract}

\section{Introduction}\label{s:introduction}

Our work starts with Jaeger's conjecture~\cite{jaeger}, which can be stated as follows.
\begin{conjecture}
  Let \(p\in \mathbb{N} \). Every \((4p+1)\)-regular \(4p\)-edge-connected graph has an orientation such that the
  in-degree of every vertex is either \(p\) or \(3p+1\).
\end{conjecture}
Such orientations are sometimes called \textit{valid orientations}, for instance in \cite{pralat15}, but we will call them \textit{Jaeger orientations}.
The $p=1$ case of Jaeger's conjecture corresponds to Tutte's nowhere-zero 3-flow conjecture (see for example~\cite[Open Problem 48]{bondymurty}).

The $p=1$ and $p=2$ cases of Jaeger's conjecture
(including Tutte's nowhere-zero 3-flow conjecture) are still open.
However, Han et al.~\cite{HAN20181} produced counterexamples to Jaeger's conjecture for all $p\geq 3$. 
Are these counterexamples rare?  Or would a typical $(4p+1)$-regular $4p$-edge-connected graph be expected to admit a Jaeger orientation?
To investigate this, we can study random $(4p+1)$-regular graphs, which are known
to be $(4p+1)$-connected (and hence $4p$-edge-connected) with high probability~\cite{wormald81}.

Let $\mathcal{G}_{n,d}$ denote the uniform model of random $d$-regular graphs on the vertex set $[n]:=\{1,2,\ldots, n\}$.
We say that an event holds asymptotically almost surely (a.a.s.) in $\mathcal{G}_{n,d}$ if the
probability that the event holds tends to~1 as $n$ tends to infinity.
Pra{\l}at and Wormald~\cite{pralat15} proved that Jaeger's conjecture holds a.a.s.\ for the case $p=1$, and the corresponding result for $p=2$ was established by Delcourt, Huq and Pra{\l}at~\cite{delcourt25}, both using the small subgraph conditioning method (SSCM). Using a completely different approach involving the Expander Mixing Lemma,  
Alon and Pra{\l}at~\cite{alon11} proved that Jaeger's conjecture holds a.a.s.\ when $p\geq p_0$, where $p_0$ is a large but unspecified constant. We conjecture that this behaviour holds for all~$p$.

\begin{conjecture}
\label{c:jaeger-aas}
For all $p\geq 1$, Jaeger's conjecture a.a.s.\ holds. That is, for any fixed $p\geq 1$, a.a.s.\ $\mathcal{G}_{n,4p+1}$ has a Jaeger orientation.
\end{conjecture}

We now define a generalisation of the orientations considered in Jaeger's conjecture.
A \emph{\(p\)-orientation} of a \(d\)-regular graph \(G\) is an orientation of \(G\) such that every vertex has in-degree either \(p\) or \(d-p\).  (Equivalently, every vertex either has in-degree $p$ or out-degree $p$.) A $p$-orientation is a Jaeger orientation if $d=4p+1$.  We will use the small subgraph conditioning method, and other arguments, to study the a.a.s.\ existence of $p$-orientations in $\mathcal{G}_{n,d}$. As we see in Section~\ref{ss:bisection}, if a graph has a $p$-orientation then it has an even number of vertices. Hence our asymptotics are as $n\to\infty$ along the positive even integers, unless specified otherwise.

Our main result is the following.
\begin{theorem}
  \label{thm:small-dp}
  Let $d\geq 3$ and $p\geq 1$ be constants.
  \vspace*{-0.5\baselineskip}
\begin{itemize}
\item[\emph{(i)}] If $\binom{d}{p} \leq  2^{d/2-1}$ then a.a.s.\ 
   \(G\in\mathcal{G}_{n,d}\) has no \(p\)-orientation.
\item[\emph{(ii)}]  If $p\in \{1,2,3,4\}$ and 
\begin{equation}
    \label{eq:inequality}
    \big(d^2-4dp+4p^2-d\big)^2 < d^2(d-1)
\end{equation}
  then a.a.s. \(G\in\mathcal{G}_{n,d}\) has a \(p\)-orientation.
\end{itemize}
\end{theorem}

In particular, this theorem implies that even though counterexamples to Jaeger's conjecture exist when $p\in \{3,4\}$, see~\cite{HAN20181},  they are very rare.

\begin{corollary}
Jaeger's conjecture holds a.a.s.\ when $p=3$ or $p=4$.
That is, $\mathcal{G}_{n,13}$ a.a.s.\ has a 3-orientation and $\mathcal{G}_{n,17}$ a.a.s.\ has a 4-orientation.
\end{corollary}

\begin{table}[ht!]
\begin{center}
\renewcommand{\arraystretch}{1.2}
  \begin{tabular}{|c|c|c|c|c|c|c|c|c|c|c|c|c|c|c|c|c|c|c|}
  \hline
    \backslashbox{\(p\)}{\(d\)} & 3 & 4 & 5 & 6 & 7 & 8 & 9 & 10 & 11 & 12 & 13 & 14 &  15 & 16 & 17 & 18 & 19 & 20\\ \hline
    1 & Y & Y & \cellcolor{gray!15}Y & Y & ? & N & N & N & N & N & N & N & N & N & N & N & N & N \\ \hline
    2 & Y & Y & Y & Y & Y & Y & \cellcolor{gray!15}Y & Y & ? & ? & ? & N${}^*$ & N${}^*$ & N & N & N & N & N\\ \hline
    3 & $\blacksquare$  & Y & Y & Y & Y & Y & Y & Y & Y & Y & \cellcolor{gray!15}Y & Y & ? & ? & ? & ? & N${}^*$ & N${}^*$\\ \hline
    4 & $\blacksquare$  & $\blacksquare$ & Y & Y & Y & Y & Y & Y & Y & Y & Y & Y & Y & Y & \cellcolor{gray!15}Y & ? & ? & ?\\ \hline
  \end{tabular}
\caption{Summary of our results: ``Y'' means that a.a.s.\ $\mathcal{G}_{n,d}$ has a $p$-orientation; ``N'' means that a.a.s.\ $\mathcal{G}_{n,d}$ has no $p$-orientation, by Theorem~\ref{thm:small-dp}(i); ``N${}^*$'' means that Theorem~\ref{thm:small-dp}(i) does not apply but a.a.s.\ $\mathcal{G}_{n,d}$ does not have large enough maximum bisection size to have a $p$-orientation (see Section~\ref{ss:bisection}). A black square denotes infeasible parameters.  The highlighted (grey) cell in each row corresponds to the Jaeger conjecture. } 
\label{t:results}
\end{center}
\end{table}

The positive results from Theorem~\ref{thm:small-dp}(ii) are displayed in Table~\ref{t:results}, together with some relevant negative results. 
We see from Table~\ref{t:results} that for $p=1,2,3$, the value of $d$ from the Jaeger conjecture ($d=4p+1)$ is not the largest value of $d$ for which $\mathcal{G}_{n,d}$ a.a.s.\ has a  $p$-orientation.  (See end of Section~\ref{ss:bisection} for more discussion.)

We will prove some key positive results (see Lemma~\ref{lem:list})
using the SSCM, working in the configuration model.
Specifically, we will analyse the random variable $Y$ defined
to be the number of $p$-orientations of a random configuration.
The inequality (\ref{eq:inequality}) in Theorem~\ref{thm:small-dp}(ii) is a technical condition which is required
for the SSCM to apply to the random variable $Y$.

We know from~\cite{alon11} that Jaeger's condition holds a.a.s.\ for all $p$ sufficiently large. Indeed, Alon and Pra{\l}at's argument from~\cite{alon11} can be easily modified 
(in particular by using~\cite[Theorem~2.3]{LSWZ} rather than~\cite[Corollary~2.2]{LSWZ})
to prove that there exists some sufficiently large positive integer $d_0$ such that
$\mathcal{G}_{n,d}$ a.a.s.\ has a $p$-orientation for all fixed $d\geq d_0$
and $(d-1)/4\leq p\leq d/2$.

For these reasons, we restrict our attention to $p\in \{1,2,3,4\}$ and $d\leq 20$, other than some brief remarks at the end of Section~\ref{ss:bisection}.
However, we conjecture that $\mathcal{G}_{n,d}$ a.a.s.\ has a $p$-orientation for any $d\geq 3$ and $p$ which satisfy (\ref{eq:inequality}).
Indeed, we perform almost all calculations for the SSCM in
generality, with the exception of the proof of Lemma~\ref{eqn:hard-phi-maximum-lemma}.

The structure of the paper is as follows.
After discussing connections with maximum bisection size in Section~\ref{ss:bisection}, the configuration model and the SSCM are introduced in Section~\ref{s:preliminaries}. 
The expected number of $p$-orientations in the configuration
model is calculated in Lemma~\ref{lem:the-first-mom}. Combining this with a special case analysed in Lemma~\ref{l:dp81} establishes
Theorem~\ref{thm:small-dp}(i).  The remainder of the paper
is devoted to using the SSCM to prove that a.a.s\ $\mathcal{G}_{n,d}$
has a $p$-orientation for all
$(d,p)$ in \eqref{eq:list}.
Working in the configuration model, the impact of short cycles is analysed in Section~\ref{s:short-cycles} and
the second moment is investigated in Section~\ref{s:mom2}, with  the proof of some
technical results deferred to Appendix~\ref{s:hard}.
The proof of Theorem~\ref{thm:small-dp}(ii) can be
found at the end of Section~\ref{ss:laplace}.

Before continuing, we make some elementary observations.
If $d=2p$ then a $p$-orientation is an Eulerian orientation, which exists for every $d$-regular graph. Here (\ref{eq:inequality}) holds and so Theorem~\ref{thm:small-dp}(ii) is true (deterministically) when $p=d/2$. Hence we can assume that $d\neq 2p$. Furthermore
 every $(d-p)$-orientation is also
a $p$-orientation, so we may assume that $p < d/2$.  
It is not clear whether the property of a.a.s.\ existence of $p$-orientations is monotonically decreasing with respect to $d$, though this seems likely.

\subsection{Connections with maximum bisection size}\label{ss:bisection}

If a graph $G$ has an even number of vertices then a \emph{bisection} is a partition of the vertices into two sets of equal size, and the \emph{size} of the bisection is the number of edges in $G$ which go between the two parts.
The maximum bisection size in $G$ is the maximum size of any bisection of $G$,
which equals the maximum number of edges in any bipartite subgraph of $G$.

Now assume that $p<d/2$ and fix a $p$-orientation of a graph $G$ on vertex set $[n]$. Let $U$ be the set of vertices with in-degree $p$ and let $W$ be the set of vertices with out-degree $p$.  Then $|U|=|W|$, as can
be seen by double-counting the edges of $G$, first by in-degree and then by out-degree.
Hence a $p$-orientation determines a bisection of $G$, and in particular,
if $G$ has a $p$-orientation then $G$ has an even number of vertices.
The maximum number of edges which can be oriented from $W$ to $U$ is $pn/2$, which
is the total in-degree capacity of $U$.  Hence at least $(d-p)n/2 - pn/2 = (d-2p)n/2$ edges must be
oriented from $U$ to $W$.
Therefore, if $\mathcal{G}_{n,d}$ a.a.s.\ has a $p$-orientation and $p < d/2$ then
the maximum bisection size of $\mathcal{G}_{n,d}$ is a.a.s.\ at least $nd/2-pn$, which is a $1-2p/d$ fraction
of the total number of edges.

Coja-Oghlan et al~\cite[Corollary~1.3]{MaxBisection} gave an formula for an a.a.s.\ upper bound on the maximum cut size of 
$\mathcal{G}_{n,d}$, presented as a fraction of the total number of edges $dn/2$.
In particular, this gives an a.a.s.\ upper bound on the maximum bisection size of $\mathcal{G}_{n,d}$. We can use their formula to show that a.a.s.\ $\mathcal{G}_{n,d}$ has no
$p$-orientations for the following pairs $(d,p)$:

\vspace*{-0.5\baselineskip}
\begin{itemize}
\item $(14,2)$, since $1-4/14 \approx 0.7142$ is larger than the a.a.s.\ upper bound  $0.7028$ for $\mathcal{G}_{n,14}$.
\item $(15,2)$, since $1-4/15 \approx 0.7333$  is larger than the a.a.s.\ upper bound $0.6965$ for $\mathcal{G}_{n,15}$.
\item $(19,3)$, since $1-6/19 \approx 0.6842$ is larger than the a.a.s.\ upper bound  $0.6749$ for $\mathcal{G}_{n,19}$.
\item $(20,3)$, since $1-6/20 =0.7$ is larger than the a.a.s.\ upper bound $0.6703$ for $\mathcal{G}_{n,20}$.
\end{itemize}
These negative results do not follow from our first moment argument
Theorem~\ref{thm:small-dp}(i).  In particular, when $p=2$ the smallest value of $d$ for which we can deduce
a negative result from Theorem~\ref{thm:small-dp}(i) is $d=16$, while when $p=3$ the smallest value of $d$ for which we can deduce
a negative result from Theorem~\ref{thm:small-dp}(i) is $d=24$. 

A recent paper by Harangi~\cite{harangi-2RSB} gave 2-RSB upper bounds on the maximum cut size for $d=3,\ldots, 7$,
which improve on those from~\cite{MaxBisection}, together with a link to SAGE code for investigations for higher degrees. The improved bound for $d=7$ does not imply a negative result for the case $(d,p)=(7,1)$, and unfortunately we were not able to
resolve any unknown cases in Table~\ref{t:results} using the code.

It is possible that the a.a.s.\ existence of $p$-orientations in $\mathcal{G}_{n,d}$ is essentially determined by this connection with maximum bisection size.
Let $p^*(d)$ denote the minimum value of $p\in \{1,2,\ldots, \lfloor d/2\rfloor \}$ such that $\mathcal{G}_{n,d}$ a.a.s.\ has a $p$-orientation.
Further, define 
\[ p_{\text{bis}}(d):= \min \big\{\,  p \mid  1\leq p < d/2\,\,\, \text{and a.a.s.}\,\, \mathcal{G}_{n,d}\,\,\,  \text{has a bisection of size } \geq (d-2p)n/2\big\}.\]
We make the following conjecture, which is consistent with our data.

\begin{conjecture}
\label{our-conjecture}
For all $d\geq 3$ we have $p^*(d)\in \{ p_{\text{\emph{bis}}}(d),\, p_{\text{\emph{bis}}}(d) + 1\}$. 
\end{conjecture}

If (\ref{eq:inequality}) does not hold then SSCM method cannot be applied to the random variable $Y$ (defined in Section~\ref{s:SSCM}), because the variance of $Y$ becomes too large. 
One approach that could be used to tackle Conjecture~\ref{our-conjecture} would be to apply the SSCM to a different
random variable with smaller variance. 
Alternatively, it may be possible to further exploit the connection between maximum bisection size and existence of $p$-orientations, which is reminiscent of the
connection between the independence ratio and existence of star decompositions discussed in~\cite{stars,harangi1}.
In particular, it might be possible to adapt some of the ideas from
Harangi~\cite{harangi1, harangi2} or Gerencs{\' e}r and Harangi~\cite{GH} on the connection between independence number,
orientations and cuts.

\bigskip

To conclude this section we discuss the behaviour of $p^*(d)$ for large~$d$.
It follows from the proof of Bollob{\' a}s~\cite[Theorem~1]{bollobas1988} that the
expected number of bisections of size $(1-\eta) dn/4$ in $\mathcal{G}_{n,d}$
tends to zero whenever
\begin{equation}
\label{eq:bollobas}
 (1-\eta)\log(1-\eta) + (1+\eta)\log(1+\eta) > \frac{4\log 2}{d}.
\end{equation}
Here and throughout the paper, $\log$ denotes the natural logarithm.
(Note that Bollob{\' a}s works in the configuration model but the conclusion for $\mathcal{G}_{n,d}$
holds using standard arguments, see Section~\ref{s:preliminaries} and (\ref{eqn:pndsimple}).)
The left hand side of (\ref{eq:bollobas}) is bounded below by $\eta^2$ (and when $\eta$ is
small, $\eta^2$ is a good approximation).
Applying
Markov's Lemma implies that a.a.s.\ no bisections of size $(1-\eta)dn/4$ exist when
$\eta^2 > \frac{4\log 2}{d}$.
Setting $(1+\eta)/2 = 1-2p/d$, from this we can conclude that 
\[ p^*(d) > \frac{d}{4} - \frac{\sqrt{d\log 2}}{2},\]
at least when $d$ is large. 
Combining this with the analogue of Alon and Pra{\l}at's result from~\cite{alon11}, discussed earlier, this implies that
\[ \lim_{d\to\infty} \frac{p^*(d)}{d} = \frac{1}{4},\]
which asymptotically matches the Jaeger case $d=4p+1$.

\section{Preliminaries}\label{s:preliminaries}

We say that the pair $(d,p)$ is \emph{affirmative} if a.a.s.\ 
  \(G\in\mathcal{G}_{n,d}\) has a \(p\)-orientation.  First we show how to reduce
 Theorem~\ref{thm:small-dp}(ii) to a small number of cases.

\begin{lemma}\label{lem:list}
Theorem~\emph{\ref{thm:small-dp}(ii)} holds if the following
pairs are affirmative:
\begin{equation}
\label{eq:list}
(d,p)\in\{(6,1),\, (10,2),\, (13,3),\, (14,3),\, (17,4)\}.  
\end{equation}
\end{lemma}

\begin{proof}
It is easy to check that for $p\in \{1,2,3,4\}$, the only pairs $(d,p)$ which satisfy (\ref{eq:inequality}) are those which have a `Y' entry in Table~\ref{t:results}.
First we prove that $(d,p)=(3,1)$ is affirmative. We know that $\mathcal{G}_{n,3}$ 
a.a.s.\ contains a Hamilton cycle~\cite{RW92}. Given a cubic graph $G$ with a Hamilton cycle $H$,
we can orient $H$ to form a directed Hamilton cycle, then orient the perfect matching $G \setminus H$ arbitrarily, giving a 1-orientation of $G$.  Hence $(d,p)=(3,1)$ is affirmative.  
Similarly, the edges of $\mathcal{G}_{n,4}$ can a.a.s.\ be partitioned into two Hamilton cycles~\cite{KW},
and as $n$ is even we can orient one Hamilton cycle to give a directed cycle, and orient the other in an
alternating fashion so that half of the vertices are sinks and half are sources. Overall this
shows that $(4,1)$ is affirmative.
Also note that $(5,1)$ is affirmative~\cite{pralat15} and $(9,2)$ is affirmative~\cite{delcourt25}.

All Eulerian cases $(2p,p)$ are affirmative, as can be seen by orienting an Eulerian cycle.  
By deleting a perfect matching from a $(2p+1)$-regular graph we obtain an Eulerian graph,
and random regular graphs with degree at least three a.a.s\ contain a perfect matching~\cite{RW94},
so the $(2p+1,p)$ case is affirmative. 

Finally, we know that $\mathcal{G}_{n,d+2}$ is contiguous with the superposition of a random Hamilton
cycle and $\mathcal{G}_{n,d}$ when $d\geq 3$, see~\cite{janson95}.  If $\mathcal{G}_{n,d}$ a.a.s.\ has a $p$-orientation
then, by orienting the Hamilton cycle we obtain a $(p+1)$-orientation of $\mathcal{G}_{n,d+2}$ a.a.s..  Hence $(d+2,p+1)$ is affirmative whenever $(d,p)$ is affirmative. These arguments, together with the fact that $(d,p)$ is affirmative if and only if $(d,d-p)$ is affirmative, allow us to deduce all remaining affirmative cases in Table~\ref{t:results}.  
\end{proof}

We will perform our calculations in the configuration model. Take $dn$ points arranged
into $n$ cells, where each cell contains $d$ points.  All points are labelled,
and the cells are labelled with $1,\ldots, n$.  We assume that $dn$ is even.
A \emph{pairing} is a partition of the $dn$ points into pairs.  Let $\pairing_{n,d}$
denote the set of all such pairings, and let
$\mathcal{P}_{n,d}$ denote the uniform model on $\pairing_{n,d}$. 
We write
$F\in\mathcal{P}_{n,d}$ to denote that $F$ is chosen uniformly at random from $\pairing_{n,d}$.
By shrinking each cell to a vertex, a pairing $F$ gives rise to a $d$-regular multigraph
$G(F)$, which may have loops or multiple edges. Say that $F$ is simple if $G(F)$ is
simple (no loops or multiple edges).
Conditioned on being simple,
$G(F)$ is uniformly random over $\mathcal{G}_{n,d}$ when $F\in \mathcal{P}_{n,d}$.
The number of ways to pair up $2a$ points is
\begin{equation}
\label{eq:M-def}
    M(2a):= \frac{(2a)!}{a!\, 2^a},
\end{equation}
so $|\pairing_{n,d}| = M(dn)$.

Let ``Simple'' denote the set of simple pairings in $\pairing_{n,d}$.
Bender and Canfield \cite{bandc78}
proved that for constant $d$, as \(n\rightarrow \infty \),
\[\Pr (\text{Simple})\sim \exp\left(\frac{1-d^2}{4}\right).\]
Let $Y$ be a random variable on $\pairing_{n,d}$ which satisfies $Y(F) = Y(F')$ for all $F,F'\in\pairing_{n,d}$ such that $G(F) = G(F')$.  The corresponding random variable on
graphs, $Y_{\mathcal{G}}$, is defined by $Y_{\mathcal{G}}\big(G(F)\big) = Y(F)$. 
Using the fact that every $G\in\mathcal{G}_{n,d}$ corresponds to exactly the same 
number of pairings (namely, $(d!)^n$ pairings), we have
\begin{equation}
  \label{eqn:pndsimple}
  \Pr(Y_{\mathcal{G}}=0) = \Pr\big(Y=0\mid \text{Simple}\big) \leq
  \frac{\Pr(Y=0)}{\Pr (\text{Simple})} = O\big(\Pr(Y=0)\big).
\end{equation}
Therefore, to prove that a.a.s.\ $Y_{\mathcal{G}} > 0$ in $\mathcal{G}_{n,d}$, it 
suffices to prove that a.a.s.\ $Y>0$ in $\mathcal{P}_{n,d}$.  We will follow this approach
and study the number of $p$-orientations of a random pairing.

\subsection{The small subgraph conditioning method}\label{s:SSCM}

We will use the small subgraph conditioning method in the following form, adapted slightly from~\cite[Theorem~4.1]{wormald99}.  
\begin{theorem}
  \label{thm:sscm}
  Let \(\lambda _i>0\) and \(\delta _i\ge -1\), \(i=1,2,\dots,\) be real numbers and suppose that
  for each \(n\) there are random variables \(X_i=X_i(n)\), \(i=1,2,\dots,\) and
  \(Y=Y(n)\) defined on the same probability space \(\mathcal{G} =\mathcal{G} (n)\) such that
  \(X_i\) is nonnegative, integer-valued, \(Y\) is nonnegative and
  \(\Exp Y>0\) (for \(n\) sufficiently large). Suppose, further that
  \begin{enumerate} 
  \item[\emph{(A1)}] For each \(k\ge 1\), the random variables \(X_1,\dots,X_k\) are asymptotically independent
    and Poisson
    such that \({\Exp}X_i\rightarrow \lambda _i\) for all \(i\in \mathbb{N} \).
  \item[\emph{(A2)}]
    \[
    \frac{\Exp \left[Y(X_1)_{j_1}(X_2)_{j_2}\dots(X_k)_{j_k}\right]}{\Exp Y}\rightarrow \prod _{i=1}^k\big(\lambda _i(1+\delta _i)\big)^{j_i}
    \]
    for every finite sequence \(j_1,j_2,\dots,j_k\) of nonnegative integers.
  \item[\emph{(A3)}] \(\sum _{i=1}^{\infty}\lambda _i\delta _i^2<\infty \).
  \item[\emph{(A4)}]
    \[
    \frac{\Exp (Y^2)}{(\Exp Y)^2}\le \exp\left(\sum _{i=1}^{\infty}\lambda _i\delta _i^2\right)+o(1)\quad\text{as }n\rightarrow \infty .
    \]
  \end{enumerate}
  Then
  \[\Pr (Y_n>0)=\exp\left(-\sum _{\delta _i=-1}\lambda _i\right)+o(1).\]
\end{theorem}

We will apply this method with  \(Y\) given by the number of \(p\)-orientations of
\(F\in \mathcal{P} _{n,d}\) and where \(X_i\) is the number of \(i\)-cycles in \(\mathcal{P} _{n,d}\) for all $i\geq 1$.  (A set of pairs is called an $i$-cycle if the graph induced by this set of pairs form an $i$-cycle.)
Bollob{\' a}s \cite{BOLLOBAS1980311}, proved that the random variables
\(X_1,X_2,\dots\) are asymptotically independent and Poisson, with
such that $\lim_{n\rightarrow \infty }\Exp (X_i) = \lambda_i$ where
\begin{equation}
    \label{eq:lambda-def}
\lambda_i:=\frac{(d-1)^i}{2i}.
\end{equation}
Hence assumption (A1) holds.  In the next section we establish (A2) and (A3),
under the conditions of Theorem~\ref{thm:small-dp}(ii).

\section{Expected value and short cycles} \label{s:first-steps}

Fix natural numbers \(d,p\in \mathbb{N} \) with $d\geq 4$ and $1\leq p < d/2$. 

The pairs of a pairing $F$ can be oriented by specifying one point of each pair as the \emph{in-point} and the other as the \emph{out-point}.  Such an assignment of points is called a $p$-\emph{orientation} if it corresponds to a $p$-orientation of $G(F)$, where this notion is extended to multigraphs in the natural way. 
Given a pairing $F\in\pairing_{n,d}$ and a $p$-orientation of that pairing, say that a cell containing $p$ in-points is an \emph{in-vertex} and that a cell
containing $p$ out-points is an \emph{out-vertex}.
(The terminology is consistent with Delcourt et al.~\cite{delcourt25}.)
There are $n/2$ vertices of each type, as can be seen by equating the number of in-points and out-points.

The $p$ in-points in an in-vertex are called \emph{special points}, and similarly the $p$ out-points in an out-vertex are \emph{special points}.

Recall that $Y$ denotes the number of $p$-orientations of $F\in\mathcal{P}_{n,d}$.

\begin{lemma}
\label{lem:the-first-mom}
Suppose that $d\geq 4$ and $1\leq p < d/2$.  Then
\begin{equation}
  \label{eq:EY}
  \Exp (Y)=\frac{\binom{n}{n/2}{\binom{d}{p}}^n(dn/2)!}{M(dn)} 
  \sim \sqrt{d}\left(2^{1-d/2}\binom{d}{p}\right)^n.
\end{equation}
In particular, Theorem~\emph{\ref{thm:small-dp}(i)} holds when $2^{1-d/2}\binom{d}{p} < 1$.
\end{lemma}

\begin{proof}
Observe
that \(|\mathcal{P} _{n,d}|\, \Exp  Y\) counts the number of ways to choose a pairing together with a $p$-orientation of that pairing. There are
\(\binom{n}{n/2}\) ways to specify the in-vertices
and out-vertices.  For each cell there are \(p\) special points,
so there are \({\binom dp}^n\) ways of choosing all of the special
points.  Finally, since the in-points must be matched to the out-points,
there are \((dn/2)!\) possible ways of forming a pairing \(F\) from a
choice of all in-vertices, out-vertices and all special points.
This proves the combinatorial expression for $\Exp (Y)$, recalling that $|\mathcal{P}_{n,d}| = M(dn)$,
and the asymptotic
expression follows by applying Stirling's formula.

Finally, suppose that $2^{1-d/2}\binom{d}{p} < 1$. Then (\ref{eq:EY}) implies that $\Exp(Y) = o(1)$, and the second statement of the lemma follows by applying Markov's inequality. 
\end{proof}

It follows from (\ref{eq:EY}) that $\Exp(Y)$ tends to infinity whenever $2^{1-d/2}\binom{d}{p}>1$.  This occurs for all $(d,p)$ given in (\ref{eq:list}), as can be checked by substitution.
We now show that there is a unique pair $(d,p)$ with $1\leq p < d/2$ for which $\Exp(Y)$ is finite and positive, and complete the proof of Theorem~\ref{thm:small-dp}(i).

\begin{lemma}
\label{l:dp81}
  The identity $\binom{d}{p}=2^{d/2-1}$ has a unique solution with $d\geq 4$ and $1 \leq p < d/2$, namely   \((d,p)=(8,1)\). Furthermore, $F\in\mathcal{P}_{n,8}$ a.a.s.\ admits no $1$-orientation.
\end{lemma}

\begin{proof} 
 By a theorem of Erd{\H o}s (\cite[Chapter~3]{pftb})
  the binomial coefficient \(\binom dp\) is a perfect power only if \(p\le 3\).
  It is easy to check that for $p\in \{2,3\}$ there are no solutions,
  while for $p=1$ the only solution is $d=8$.

To prove the second statement, let
\(\opairing_{n,d,p}\) be the uniform probability space of ordered pairs
  \(\vec{F}=(F,O)\), where \(O\) is a $p$-orientation of \(F\in \mathcal{P} _{n,d}\).  
It follows from (\ref{eq:EY}) that 
\begin{equation}
  \label{eqn:omega-size}
  |\opairing_{n,d,p} |=\Exp (Y)|\mathcal{P} _{n,d}|=\binom{n}{n/2}\binom{d}{p}^n\left(\frac{dn}2\right)!.
\end{equation}
A directed cycle in $\vec{F}\in\opairing_{n,d,p}$ is a 
subpairing corresponding to a cycle \(C\) oriented such that
every vertex of \(C\) has in-degree \(1\) in \(C\). 
For \(r\in \mathbb{N} \),
let \(W_r\) be the number of directed
cycles of length \(r\) in \(\vec{F}\in \opairing_{n,d,p} \) and let \(W=\sum _{r=1}^{\infty}W_r\)
be the number of directed cycles in \(\vec{F}\in \opairing_{n,d,p} \).

\noindent {\bf Claim:}\ We claim that for any fixed positive integer $k$ we have a.a.s\ $W\geq k$.
Observe that 
\begin{align*}
\Exp (W_r) &= \frac{1}{|\opairing_{n,d,p}|}\, \frac{(n)_r}{r}\, \big( d(d-1)\big)^r\,
 \binom{n}{n/2}\, \binom{d-2}{p-1}^r\, \binom{d}{p}^{n-r}\, (dn/2-r)!\\
  &\sim \frac{1}{r}\left(\frac{2p(d-p)}{d}\right)^r =: \mu_r.
\end{align*}
To see this, note that the cells in a directed $r$-cycle can be chosen
in $(n)_r/r$ ways, then points selected for the $r$-cycle in $\big( d(d-1)\big)^r$ ways, one as an in-point and one as an out-point.  The in-vertices can be
specified in $\binom{n}{n/2}$ ways, and remaining special points from cells in the $r$-cycle can be chosen in $\binom{d-2}{p-1}$ ways, or $\binom{d}{p}$ ways for cells outside the $r$-cycle. Finally, there are $(dn/2-r)!$ ways to pair up the remaining in-points and out-points, and we divide by (\ref{eqn:omega-size}) to
give an expected value.
Similar calculations show that $\Exp[ (W_r)_2 ]\sim \mu_r^2$, and hence that
$\textbf{Var}(W_r)\sim \mu_r$. (In fact, the $W_r$ are asymptotically independent Poisson random variables with mean $\mu_r$, though we do not need this here.)
Since \(\mu_r\to\infty\) as \(r\to\infty\), we can assume that
  \(k<\mu_r/2\). Then, Chebyshev's Inequality implies that
  \[ \lim_{n\to\infty} \Pr(W_r\le k)\le \lim_{n\to\infty}  \Pr\left(|W_r-\mu_r|\ge \frac{\mu_r}{2}\right)\le\frac{4}{\mu_r},\]
  hence,
  \[ \lim_{n\to\infty}  \Pr(W>k)\ge 1- \lim_{n\to\infty}  \Pr(W_r\le k)\ge 1-\frac{4}{\mu_r}.\]
  The proof of the claim is completed by taking \(r\) arbitrarily large.

  Finally, fix a positive integer $k$ and let $(d,p)=(8,1)$.
  Let \(O\) be an orientation of \(F\in \mathcal{P}_{n,d}\) such that
  \(W(F,O)>k\). 
  Then, reversing the direction of the edges along different \(k\) cycles, we can
  obtain at least \(k\) further \(p\)-orientations of \(F\). Thus,
  if \(Y(F)\le k\), then \(W(F,O)\le k\) for any orientation \(O\).
  Applying this observation we see that
  \[ |\{F\in \mathcal{P} _{n,8} \mid  1\le Y(F)\le k\}|
  \le  |\big\{\vec{F}\in \opairing_{n,8,1}  \mid  W\big(\vec{F}\big)\le k\big\}|.\]
  However, the left hand side of this inequality
  is \(|\mathcal{P} _{n,8}|\, \Pr \left(1\le Y(F)\le k\right)\)
  and the right hand side is \(|\opairing_{n,8,1} |\, \Pr \left(1\le W(\vec{F})\le k\right)\). So, applying our claim gives
  \[\Pr \left(1\le Y(F)\le k\right)=o\left(\frac{|\opairing_{n,8,1} |}{|\mathcal{P} _{n,8}|}\right)=o(1),\]
  since \(|\opairing_{n,8,1} |/|\mathcal{P} _{n,8}|=\Exp (Y)=O(1)\).
  Hence, by Markov's inequality,
  \[\limsup_{n\rightarrow \infty }\Pr (Y>0)=\limsup_{n\rightarrow \infty }\Pr (Y>k)\le \lim_{n\rightarrow \infty }\frac{\Exp (Y)}{k}=\frac{\sqrt{8}}{k}.\]
  Since this holds for arbitrary \(k\in \mathbb{N} \), the result follows by taking  \(k\)  to infinity.
\end{proof}

\subsection{The effect of short cycles} \label{s:short-cycles}

In order to establish that condition (A2) of Theorem~\ref{thm:sscm} holds,
we must investigate how the presence of short cycles affects the number of
$p$-orientations. For the remainder of the paper we assume that $\Exp Y\to\infty$, which in particular holds for all $(d,p)$ in \eqref{eq:list}.

\begin{lemma}
  \label{lemma:joint-moments}
  Suppose that $d,p$ are fixed positive integers such that  \(d\ge 4\) and $p < d/2$. Then for all positive integers $i$,
  \[\frac{\Exp (YX_i)}{\Exp (Y)}\rightarrow \lambda _i(1+\delta _i)\]
  as \(n\rightarrow \infty \), where \(\delta _i\) is defined by
  \begin{equation}
      \label{eq:delta-def}
  \delta _i\coloneq \left(-\frac{d^2-4dp+4p^2-d}{d(d-1)}\right)^i.
  \end{equation}
  Furthermore, $\delta_i>-1$ for all $i\geq 1$.
\end{lemma}

\begin{proof}
  Firstly, observe that \(|\mathcal{P} _{n,d}\, |\Exp (YX_i)\) counts the number of triples of
  the form \((F,C,O)\), such \(F\in \mathcal{P} _{n,d}\) contains \(C\)
  an \(i\)-cycle in \(\mathcal{P} _{n,d}\) and \(O\) is a \(p\)-orientation of \(F\).
  There are
  \[\frac{(n)_i\, \big(d(d-1)\big)^i}{2i}\]
  ways of to choose the cells and pairs of an $i$-cycle \(C\). 

  Next we will sum over all ways to orient the pairs of $C$. This will involve the following parameters:  
  \begin{itemize}
      \item 
  $j$ denotes the number of cells with in-degree 2 in $C$ (note that, by symmetry, the number of cells with  out-degree \(2\) in \(C\) also equals $j$);
\item $s$ denotes the number of cells in $C$ with in-degree $2$ in $C$ which are in-vertices;
\item $t$ denotes the number of cells of $C$ with out-degree $2$ in $C$ which are in-vertices.
    \end{itemize}
There are exactly \(2\binom{i}{2j}\) possible orientations of \(C\)
  with \(j\) cells of in-degree \(2\), since the directions of the edges
  in the cycle are determined entirely by the choice of cells of in-degree
  \(2\) or out-degree \(2\) in \(C\)
  which must alternate around the cycle.
Then there are \(\binom js\) ways of choosing the in-vertices among
  those cells of in-degree \(2\) in \(C\), and \(\binom jt\) ways of choosing
  the in-vertices
  among those cells of out-degree \(2\) in $C$.

  If an in-vertex \(v\) has in-degree \(2\) in \(C\), then two
  of the special points for \(v\) are already determined. Hence, there are
  \(\binom{d-2}{p-2}^s\) ways of choosing the remaining
  special points
  for those vertices in \(C\) which have in-degree \(2\) in \(C\) and are also in-vertices
  in \(F\). Similarly, there are \(\binom{d-2}{p-2}^t\) ways of choosing
  the remaining special points of those cells in \(C\) which have out-degree
  \(2\) in \(C\) and are also out-vertices.
  Moreover, there are \(\binom{d-2}{p}^{j-s}\) ways of choosing the special
  points of those cells in \(C\) which are out-vertices in \(F\) yet
  have in-degree \(2\) in \(C\), since none of their special points are yet determined.
  Similarly, there are \(\binom{d-2}{p}^{j-t}\) ways of choosing the special points
  for those cells in \(C\) which are out-vertices in \(F\) yet have out-degree \(2\)
  in \(C\).

  Now, observe that if a cell has in-degree equal to \(1\) in \(C\), then
  one of its special points is already determined. Hence, there are
  \(\binom{d-2}{p-1}^{i-2j}\) ways of choosing the remaining special points
  for cells in \(C\).
  
  So far, exactly \(2j\) cells
  have been determined as either in-vertices
  or out-vertices. Since \(i,s,t=O(1)\) as \(n\rightarrow \infty \) and the number of
  in-vertices in \(F\) is exactly equal to the number of out-vertices in \(F\),
  there are 
  \[ \binom{n-2j}{n/2-O(1)} \sim \binom{n}{n/2}\, 2^{-2j}\] 
  ways of choosing the remaining in-vertices.
  Note, this also determines the out-vertices of \(F\). The special points of exactly
  \(i\) cells, namely, the cells of the cycle \(C\), have been determined already.
  Hence there are \(\binom{d}{p}^{n-i}\) ways of determining the remaining
  special points.
Now that all special points have been determined, there are 
  \[ \big(dn/2-i\big)! \sim \big(dn/2)!\, \left(\frac{2}{dn}\right)^i\] 
  ways of pairing up the in-points and out-points not involved in the cycle $C$.

Let \(\gamma_k:=\binom{d-2}{p-k}\) for $k=0,1,2$.
  Combining everything together, dividing by $M(dn)$ and summing over all possible values of \(j\), \(s\), and \(t\), we obtain, using (\ref{eq:EY}),
  \begin{align*}
\Exp (YX_i)
   &\sim \sum _{j=0}^{i/2}\, \sum _{s,t=0}^j \frac{(n)_i\, \big(d(d-1)\big)^i}{i\, M(dn)} \, \binom{i}{2j}\, 
    \binom js \binom jt \, \binom{d-2}{p-2}^{s+t}\, \binom{d-2}p^{2j-(s+t)}
    \binom{d-2}{p-1}^{i-2j} \\
    & \hspace*{2cm} \times 
    \binom{n}{n/2}\, 2^{-2j}\, \binom dp^{n-i}\, \big( dn/2\big)!\, \left(\frac{2}{dn}\right)^i \\ 
    &= 2\lambda _i\, \Exp(Y) \,   \left(\frac{2p(d-p)}{d(d-1)}\right)^i\, 
\sum _{j=0}^{i/2}
    \binom{i}{2j} \left(\frac{\gamma_0}{2\gamma_1}\right)^{2j} \, \sum _{s,t=0}^j \binom js \binom jt \gamma _2^{s+t}\gamma _0^{-(s+t)}\\
    &=  2\lambda _i \, \Exp(Y)\,  \left(\frac{2p(d-p)}{d(d-1)}\right)^i
    \, \sum _{j=0}^{i/2}\binom{i}{2j} \,\left(\frac{\gamma_0}{2\gamma_1}\,\left(1+\frac{\gamma _2}{\gamma _0}\right)\right)^{2j},
  \end{align*}
  using the Binomial Theorem to evaluate the sums over $s$ and $t$. Applying the standard identity \((1+x)^m+(1-x)^m=2\sum _{k=0}^{m/2} \binom{m}{2k}x^{2k}\),
  we obtain
  \begin{align*}
    \frac{\Exp (YX_i)}{\Exp Y} &\sim 
    \lambda _i \left(\frac{2p(d-p)}{d(d-1)}\right)^i
    \left[\left(1+\frac{\gamma _0}{2\gamma _1}\left(1+\frac{\gamma _2}{\gamma _0}\right)\right)^i+\left(1-\frac{\gamma _0}{2\gamma _1}\left(1+\frac{\gamma _2}{\gamma _0}\right)\right)^i\right] \\
  &= 
    \lambda _i \left(\frac{2p(d-p)}{d(d-1)}\right)^i\left[
      \left(\frac{d(d-1)}{2p(d-p)}\right)^i
      +\left(-\frac{d^2-4dp+4p^2-d}{2p(d-p)} \right)^i
      \right] \\
    &= \lambda _i\left(1+\left(-\frac{d^2-4dp+4p^2-d}{d(d-1)}\right)^i\right).
  \end{align*}
  This completes the proof, by definition of $\delta_i$, noting that $|\delta_i| < 1$ whenever $1\leq p < d$.
\end{proof}

Using standard arguments, it is possible to generalise Lemma~\ref{lemma:joint-moments} to
account for the presence of more than one cycle, leading to Lemma~\ref{lemma:joint-moms}. The proof involves counting the number 
of ways to choose
$j_i$ distinct $i$-cycles, for $i=1,\ldots, k$, then the union
of these cycles to a full pairing and fix a $p$-orientation of the pairing.
Since the total number of vertices in these cycles is $O(1)$, the contribution
to the sum
from cycles that are not pairwise disjoint is $O(1/n)$.
We omit the details, which are very similar to the proof of~\cite[Lemma~4.3]{stars}.

\begin{lemma}
  \label{lemma:joint-moms}
  If \(j_1,\dots,j_k\) is any finite sequence of natural numbers, then
  \[\lim_{n\rightarrow \infty }\frac{\Exp \left[Y(X_1)_{j_1}(X_2)_{j_2}\cdots (X_k)_{j_k}\right]}{\Exp (Y)}
  =\prod _{i=1}^k\big(\lambda _i(1+\delta _i)\big)^{j_i}.\]
\end{lemma}     

We have shown that condition (A2) of Theorem~\ref{thm:sscm} holds.  
Next we must evaluate the following infinite sum.

\begin{lemma}
  \label{lemma:expsum-lemma}
  Let \(d\ge 3\) and \(p\) be any natural numbers. Then the sum
  \(\sum _{i=1}^{\infty}\lambda _i\delta _i^2\) converges if and only if
  condition \emph{(\ref{eq:inequality})} holds. Moreover, if
  \emph{(\ref{eq:inequality})} holds then
  \[\exp\left(\sum _{i=1}^{\infty}\lambda _i\delta _i^2\right)=d\sqrt{\frac{d-1}{d^2(d-1)-(d^2-4dp+4p^2-d)^2}}.\]
\end{lemma}

\begin{proof}
We find that
  \[\sum _{i=1}^{\infty}\lambda _i\delta _i^2=
  \sum _{i=1}^{\infty} \frac{(d-1)^i}{2i}\left(\frac{d^2-4dp+4p^2-d}{d(d-1)}\right)^{2i}
  =\frac{1}{2}\sum _{i=1}^{\infty} \frac{1}{i}\left(\frac{(d^2-4dp+4p^2-d)^2}{d^2(d-1)}\right)^i.\]
Using the fact that (\ref{eq:inequality}) holds, the result follows from the MacLaurin series expansion of \(-\log(1-z)\), which converges if and only if 
  \(z\in (-1,1)\).
\end{proof}

This lemma is the first place where we use the assumption (\ref{eq:inequality}) in order for the SSCM to apply to $Y$. 


\section{The second moment} \label{s:mom2}

Observe that
\(\mathcal{P} _{n,d}|\, \Exp (Y^2)|\) counts the number of pairings along with an ordered pair of
\(p\)-orientations. That is, this expression counts triples of the form \((F,O_1,O_2)\) where \(F\) is a pairing
and \(O_1\) and \(O_2\) are two \(p\)-orientations of \(F\), not necessarily distinct.

Suppose that we are given such a triple $(F,O_1,O_2)$.
To each cell of $F$ we associate a triple a $(j,\alpha,\beta)$, where \(0\le j\le p\) and \(\alpha ,\beta \in \{0,1\}\), as follows. Let $j$ denote the number of 
points in the cell are special in both $O_1$ and $O_2$; set $\alpha=1$ if the cell is an in-vertex in $O_1$, and otherwise let $\alpha=0$; similarly, set $\beta=1$
if the cell is an in-vertex in $O_2$, and $\beta=0$ otherwise.
We also say that the cell is a \((j,\alpha ,\beta )\)-\emph{vertex} and call the set of all such \(v\) the \((j,\alpha ,\beta )\)-\emph{class}.

Our sum will involve the following parameters:
\begin{itemize}
\item $k$ denotes the number of cells which are in-vertices in both 
\(O_1\) and \(O_2\);
\item  For $0\leq j\leq p$ and $\alpha,\beta\in \{0,1\}$, let $k_{j\alpha\beta}$
denote the number of $(j,\alpha,\beta)$-vertices.
\end{itemize}
Before proceeding, we make some quick observations.
Since the number of in-vertices is equal to the number of out-vertices
we have $0\leq k\leq n/2$.
Note that every cell which is an in-vertex in both orientations
by definition has \(\alpha =\beta =1\). Using similar arguments for the other three sums, we have
\begin{equation}
  \label{eqn:k011}
  \sum _{j=0}^p k_{j00} = k,\qquad  
   \sum _{j=0}^p k_{j01} = \frac n2 - k, \qquad 
  \sum _{j=0}^p k_{j10} = \frac n2 - k, \qquad
  \sum _{j=0}^p k_{j11} = k.
\end{equation}
Consequently, the values \(k_{0\alpha \beta }\) are determined by the values of
\(k_{j\alpha \beta }\) for \(1\le j\le p\).
Therefore, we define the following set of possible parameters
\def\disgustingk{\left(k, \left(k_{j00},k_{j01},k_{j10},k_{j11}\right)_{j\in [p]}\right)}
\def\disgustingz{\left(z, \left(z_{j00},z_{j01},z_{j10},z_{j11}\right)_{j\in [p]}\right)}
\def\jabss{{\substack{0\le j\le p \\ \alpha ,\beta \in\{ 0,1\} }}}
\def\jabsso{{\substack{1\le j\le p \\ \alpha ,\beta \in \{ 0,1\} }}}
\[
I \coloneq  \left\{ \boldsymbol{k} \in \mathbb{N} _0^{4p+1}
\mid 
0\le k\le \frac n2 \text{ and (\ref{eqn:k011}) holds}
\right\}
\] 
where \(\boldsymbol{k} =\disgustingk\).  Also let
\begin{equation}
\label{eq:hatAhatB}
 \hat{A} = \hat{A}(\boldsymbol{k}):= pn + (d-4p)k  + \sum_{j\in [p]} \, j\big(k_{j00} + k_{j11} - k_{j01} - k_{j10}\big), \quad \hat{B} = \hat{B}(\boldsymbol{k}): = \dfrac{dn}{2} - \hat{A}.
\end{equation}

\bigskip

\begin{lemma}
Suppose that $d\geq 4$ and $1\leq p < d/2$.  
Then
    \begin{equation}
  \label{eqn:y2-asympt}
  \Exp (Y^2)=\sum _{\boldsymbol{k} \in I} J_n(\boldsymbol{k} )
\end{equation}
where
\[J_n(\boldsymbol{k} )=\frac{n!}{M(dn)} \Bigg(\prod _{\jabss}\frac{\tau _j^{k_{j\alpha \beta }}}{(k_{j\alpha \beta })!}\Bigg)\hat A!\, \hat B!.\]
\end{lemma}

\begin{proof}
We must
count the number of triples \((F,O_1,O_2)\) with the parameter $\boldsymbol{k}$.
Note that since the \((j,\alpha ,\beta )\)-classes partition the cells of
a pairing \(F\in \mathcal{P} _{n,d}\), there are exactly
\begin{equation}
  \label{eqn:ways-of-classifying}
  \frac{n!}{\displaystyle{\prod _{\jabss}k_{j\alpha \beta }!}}
\end{equation}
ways to classify the vertices of \(F\).
If \(v\) is a \((j,\alpha ,\beta )\)-vertex then there are \(\binom dp\) ways of choosing
the special points of \(O_1\), followed by \(\binom{p}{j}\) ways of choosing a subset of
the special points of \(O_1\) which are special for \(O_2\) also.
Finally, there are \(\binom{d-p}{p-j}\) ways
of choosing the remaining special points of \(O_2\).
For convenience, define
\begin{equation}
\label{eq:tau}
\tau _j\coloneq \binom{d}{p}\binom{p}{j}\binom{d-p}{p-j}
\end{equation}
By definition, \(\tau _j\) counts the number of ways of choosing the special
points for \(v\) in both orientations \(O_1\) and \(O_2\) such that they have
\(j\) special points in common.
Therefore, in total, there are
\begin{equation}
  \label{eqn:tauj}
  \prod _{\jabss}\tau _j^{k_{j\alpha \beta }}
\end{equation}
ways of choosing all of the special points for \(F\)
in both orientations \(O_1\) and \(O_2\).

It remains to count the number of pairings corresponding to a choice of special points
in \(O_1\) and \(O_2\) consistent with the parameter \(\boldsymbol{k} \).
Call a point \(q\) in the pairing \(F\in \mathcal{P} _{n,d}\) an \((in,in)\)-\emph{point}
if it belongs to an in-vertex in both orientations and an
\((out,out)\)-\emph{point} if it belongs to an out-vertex in both orientations.
If \(q\) belongs to a cell which is an in-vertex in
\(O_1\) and an out-vertex in \(O_2\) we call it an \((in,out)\)-\emph{point}; if the same holds with the roles of $O_1$ and $O_2$ exchanged then we call it an (\emph{out,out})-\emph{point}.
Since in each orientation an in-point must be paired with an out-point,
each \((\text{in},\text{in})\)-point must be paired with an
\((\text{out},\text{out})\)-point and each \((\text{in},\text{out})\)-point must be paired with
an \((\text{out},\text{in})\)-point.

We count the number of \((\text{in},\text{in})\)-points corresponding
to vertices of each \((j,\alpha ,\beta )\)-class. There are four cases:
\begin{enumerate}[(i)]
\item If \((\alpha ,\beta )=(0,0)\), then an \((\text{in},\text{in})\)-point is a point which is special in
  both \(O_1\) and \(O_2\). Thus, there are \(jk_{j00}\) such \((\text{in},\text{in})\)-points.
\item If \((\alpha ,\beta )=(1,1)\), then an \((\text{in},\text{in})\)-point is a point which is special in
  neither \(O_1\) nor \(O_2\). Thus, by inclusion-exclusion, there are 
  \((d-2p+j)k_{j11}\)
  such \((\text{in},\text{in})\)-points.
\item If \((\alpha ,\beta )=(1,0)\), then an \((\text{in},\text{in})\)-point is a point which is special in
  \(O_2\) but not in \(O_1\). Thus, there are \((p-j)k_{j10}\) such \((\text{in},\text{in})\)-points.
\item If \((\alpha ,\beta )=(0,1)\), then, by symmetry with case (iii), there are
  \((p-j)k_{j01}\) such \((\text{in},\text{in})\)-points.
\end{enumerate}
Summing over all of the possibilities and applying (\ref{eqn:k011}) and (\ref{eq:hatAhatB}), we see that there are exactly
\[\sum _{j=0}^p \big(\, jk_{j00}+(d-2p+j)k_{j11}+(p-j)(k_{j01}+k_{j10})\, \big)
 = \hat{A}
\]
\((\text{in},\text{in})\)-points, and by symmetry there are the same number of (out,out)-points. Hence there are $\hat{A}!$
possible ways to pair up the \((\text{in},\text{in})\)-points with \((\text{out},\text{out})\)-points.  By symmetry, half of the remaining points must
be \((\text{in},\text{out})\)-points and the other half must be
\((\text{out},\text{in})\)-points, as these must be paired with each other. Therefore there are
$\hat B =   \frac{dn}2-\hat{A}$
\((\text{in},\text{out})\)-points and hence $\hat{B}!$
possible ways of matching \((\text{in},\text{out})\)-points with \((\text{out},\text{in})\)-points.
Now the proof of the lemma is completed by
multiplying equations (\ref{eqn:ways-of-classifying}),
(\ref{eqn:tauj}) by $\hat{A}!\, \hat{B}!$, dividing by $|\mathcal{P}_{n,d}| = M(dn)$  and summing over all  \(\boldsymbol{k}\in I \).
\end{proof}

Before proceeding with the asymptotic evaluation of this expression,
we establish a few combinatorial identities involving the quantities
\(\tau _0,\tau _1,\dots,\tau _p\) defined in (\ref{eq:tau}).
These identities follow from the fact that dividing these quantities
by $\binom{d}{p}^2$ gives
a hypergeometric distribution with parameters $(d,p,p)$, but for completeness we include a brief proof.

\begin{lemma}
  \label{lemma:tau-sums}
  For all \(s\in \mathbb{N} _0\), we have
  \[\sum _{j=0}^p(j)_s\tau _{j}= (p)_s\binom{d}{p}\binom{d-s}{p-s}.\]
  In particular, the following identities hold:
  \begin{align*}
    \sum _{j=0}^p \tau _j = \binom{d}{p}^2, \quad
    \sum _{j=0}^p j\tau _j = \frac{p^2}{d}\binom{d}{p}^2, \text{ and}\quad
    \sum _{j=0}^p j^2\tau _j = \frac{p^2(p^2-2p+d)}{d(d-1)}\binom{d}{p}^2.
  \end{align*}
\end{lemma}

\begin{proof}
  The term \((j)_s\tau _j\) counts exactly the number of ways of choosing
  the special points of a cell \(v\) for two orientations \(O_1\) and \(O_2\) such
  that they share \(j\)
  special points then choosing a sequence of \(s\) points from \(v\) which are special in both
  orientations.
  Summing over all \(j\), we count the number of ways of choosing the special points for \(v\)
  in two orientations along with a sequence of \(s\) points  which are special in both orientations.
  However, there are \(\binom{d}{p}\) ways of choosing the special
  points of \(O_1\) followed by \((p)_s\) ways of choosing the sequence of \(s\) shared special
  points then \(\binom{d-s}{p-s}\) ways of choosing the remaining special points of
  \(O_2\). This proves the first identity.  The remaining identities follow from the first.
\end{proof}

\subsection{Asymptotic expression for the second moment} \label{ss:Ey2-asympt-sum}

In this section we  consider the asymptotic behaviour of the summand \(J_n(\boldsymbol{k} )\) from 
(\ref{eqn:y2-asympt})
as \(n\rightarrow \infty \),
where \(\boldsymbol{k} \) is allowed to vary as a function of \(n\).
It is now useful to scale all our variables by a factor of $1/n$.
Given \(\boldsymbol{k} \in I\), let
\(\boldsymbol{z} =\boldsymbol{z} (\boldsymbol{k} )\coloneq \boldsymbol{k} /n\in \mathbb{R} ^{4p+1}\). We name the coordinates analogously, 
that is,
\[\boldsymbol{z} \coloneq \disgustingz=
\left(\frac kn, \left(\frac{k_{j00}}n,\frac{k_{j01}}n,\frac{k_{j10}}n,\frac{k_{j11}}n\right)_{j\in [p]}\right).
\]
Then define \(z_{0\alpha \beta }\coloneq k_{0\alpha \beta }/n\) for all \(\alpha ,\beta \in \{0,1\}\), and note that by (\ref{eqn:k011}),
\begin{equation}
  \label{eqn:implicit-z0ab}
  \sum _{j=0}^pz_{j\alpha \beta } =
  \begin{cases}
    z & \text{if } \,\, \alpha  =\beta, \\
    \tfrac{1}{2}-z & \text{if } \,\, \alpha \neq \beta.
  \end{cases}
\end{equation}
By definition, \(z_{j\alpha \beta }\) is the proportion of the vertices which are
\((j,\alpha ,\beta )\)-vertices and \(z\) is the proportion of
vertices which are in-vertices in both orientations.
Note that the range of possible values of \(\boldsymbol{z} \) lies in the set
\begin{align*}
  K&\coloneq \bigg\{
  \boldsymbol{z}  \in  \mathbb{R} ^{4p+1} \mid 
  0\le z\le 1/2,\,\,
  \sum _{j\in [p]} z_{j00},\sum _{j\in [p]} z_{j11} \le  z,\,\,
  \sum _{j\in [p]} z_{j10}, \sum _{j\in [p]} z_{j01}\le \dfrac12-z,
  \\
  &
  \hspace{6cm}
  \text{ and } \,z_{j\alpha \beta }\ge 0 \,\text{ for all } j\in [p]
  \text{ and } \alpha ,\beta \in \{0,1\}
  \bigg\}.
\end{align*}
Note that $K$ is a compact convex set with nonempty interior $K^{\circ}$.
In fact, \(\varphi \) extends continuously to the boundary of \(K\)
using the convention that \({0\log(0)=0}\).

Let \(A\coloneq \hat A/n\) and \(B\coloneq \hat B/n\), so that
\begin{equation} 
\label{eq:A-def}
A = A(\boldsymbol{z}) = p + (d-4p)z + \sum_{j\in [p]} j\big( z_{j00} + z_{j11} - z_{j01} - z_{j10}),
\quad B = B(\boldsymbol{z}) = \dfrac{d}{2}-A.
\end{equation}
These quantities depend only on \(\boldsymbol{z} =\boldsymbol{k} /n\).

\bigskip

\begin{lemma}
\label{lem:J-asymptotic}
Suppose that $d\geq 4$ and $1\leq p < d/2$.  
Define
\begin{equation}
  \label{eqn:bn-def}
  b_n\coloneq 
  \frac{(2\pi n)^{-(4p+1)/2}\, d^{-dn/2}}{\sqrt{2}}
\end{equation}
and  let \(\psi \), \(\varphi: K^\circ\to \mathbb{R}\) be functions of \(\boldsymbol{z} \) defined on the interior of
\(K\) by
\begin{equation}
  \label{eqn:psi-def}
  \psi (\boldsymbol{z} )\coloneq \sqrt{AB}\, \Bigg(\prod _{\jabss}z_{j\alpha \beta }\Bigg)^{-1/2}
\end{equation}
and
\begin{align}
  \varphi (\boldsymbol{z} )&\coloneq 
  A\log A + B\log B + \sum _{\jabss}\left(z_{j\alpha \beta }\log(\tau _j)-z_{j\alpha \beta }\log(z_{j\alpha \beta })\right).
  \label{eqn:phi-def}
\end{align}
Then
\[
J_n(\boldsymbol{k} )\sim E(\boldsymbol{k},n)\, b_n\, \psi (\boldsymbol{k} /n)\, e^{n\varphi (\boldsymbol{k} /n)}, 
\]
where $E(\boldsymbol{k},n)$ is some function which satisfies
$E(\boldsymbol{k},n) = O(1)$ for all $\boldsymbol{k}\in I$, 
and
$E(\boldsymbol{k},n) = 1 + o(1)$ whenever
$k_{j\alpha\beta}\to\infty$  for all $(j,\alpha,\beta)\in [p]\times\{0,1\}^2$.
\end{lemma}

\begin{proof}
Let $x\vee y$ denote $\max\{x, y\}$.
The result is proved by applying Stirling's formula in the form
\[ N! = \sqrt{2\pi (N\vee 1)}\, \left(\frac{N}{e}\right)^N\, \, e^{O(1/(N+1))}
\]
which is valid for all nonnegative integers.
to (\ref{eq:M-def}), and simplifying.
Note that for any \(\boldsymbol{k}\in I \) (possibly depending on \(n\)),
the error factor from Stirling's formula is always $O(1)$,
while if all factorial arguments tend to infinity then the error factor is $1+o(1)$, as claimed. 
\end{proof}

\bigskip

The function \(\varphi \) has a natural symmetry, which will be useful later:
\begin{equation}
  \label{eqn:phisymm}
  \varphi \left(z,(z_{j00},z_{j01},z_{j10},z_{j11})_{j\in [p]}\right)
  =
  \varphi \left(\dfrac{1}{2}-z,(z_{j01},z_{j00},z_{j11},z_{j10})_{j\in [p]}\right).
\end{equation}
This symmetry has a combinatorial meaning: If \(O_1\) and \(O_2\) are a
pair of orientations of \(F\in \mathcal{P} _{n,d}\)
corresponding to \(\boldsymbol{z} \), then the pair $(O_1,O_2^{-1})$ 
corresponds to the parameters
\(\boldsymbol{z} ^{-1}\coloneq \left(\frac{1}{2}-z,(z_{j01},z_{j00},z_{j11},z_{j10})_{j\in [p]}\right)\), where $O_2^{-1}$ is
the dual orientation of $O_2$, obtained by reversing the direction of every
edge.
Since taking the dual is an involution, this gives a bijection between
triples corresponding to \(\boldsymbol{z} \) and those corresponding to \(\boldsymbol{z} ^{-1}\).
\(\varphi \).

\bigskip
We define the point \(\boldsymbol{z} ^{\ast}=\left(z^{\ast},(z_{j00}^{\ast},z_{j01}^{\ast},z_{j10}^{\ast},z_{j11}^{\ast})_{j\in [p]}\right)\) as follows:
\begin{equation}
  \label{eqn:equation-summat}
  z^{\ast} \coloneq  \dfrac{1}{4} \quad\text{and}\quad z_{j\alpha \beta }^{\ast}\coloneq \frac{\tau _j}{4\binom{d}{p}^2},
\end{equation}
for all $(j,\alpha,\beta)\in [p]\times\{0,1\}^2$.
Since \(z^{\ast}=1/4=1/2-z^{\ast}\), using (\ref{eqn:implicit-z0ab}) together
with
Lemma \ref{lemma:tau-sums} implies that for any \(\alpha ,\beta \in \{0,1\}\), 
\[z_{0\alpha \beta }^{\ast}=\dfrac{1}{4}-\sum _{j\in [p]} z_{j\alpha \beta }^{\ast}=\frac{\binom{d}{p}^2-\sum _{j\in [p]} \tau _j}{4\binom{d}{p}^2}
=\frac{\tau _0}{4\binom{d}{p}^2}.\]
Hence the equation for $z_{j\alpha\beta}^\ast$ given in (\ref{eqn:equation-summat}) 
is also valid when $j=0$.

We remark that the point \(\boldsymbol{z} ^{\ast}\) has a natural combinatorial interpretation,
as it corresponds to choosing the special points for $O_1$ and $O_2$ in each cell
independently and uniformly, as well as independently and uniformly deciding for
each cell whether it is an in-vertex or out-vertex in each of $O_1$, $O_2$.
In this setting, the probability that a given cell is a 
 \((j,\alpha ,\beta )\)-vertex is given by
\[
  \Pr \left[v\text{ is a } (j,\alpha ,\beta )\text{-vertex}\right] = \frac{\tau _j}{4\binom dp^2}=z^{\ast}_{j\alpha \beta }.
\]

\bigskip

The key technical part of the second moment analysis is contained in the the following lemma, which is proved in Appendix~\ref{s:hard}.

\begin{lemma}
  \label{eqn:hard-phi-maximum-lemma}
  Suppose that $(d,p)$ is one of the pairs given in \emph{\eqref{eq:list}}.
  Then \(\varphi \) has a unique maximum on $K$, which is attained at the point \(\boldsymbol{z} ^{\ast} \in K^\circ\).
\end{lemma}

We now evaluate \(\varphi \) and \(\psi \) at the point \(\boldsymbol{z} ^{\ast}\) for future reference.
Lemma \ref{lemma:tau-sums} implies that
\(A(\boldsymbol{z} ^{\ast})=B(\boldsymbol{z} ^{\ast})=d/4,\)
and further that
\begin{align}
  \varphi (\boldsymbol{z} ^{\ast})&=\frac d2 \log\left(\frac d4\right)
  + 4 \sum _{j=0}^p \frac{\tau _j}{4\binom dp^2}\left[\log(\tau _j)-\log\left(\frac{\tau _j}{4\binom dp^2}\right)\right] \nonumber\\
  &= \frac d2 \log\left(\frac d4\right)
  + 2\log\left[2\binom dp\right]. \label{eqn:phi(zstar)}
\end{align}
Similarly,
\begin{equation}
  \label{eqn:psi-zstar}
  \psi(\boldsymbol{z}^{\ast})= \frac{d}{4} \,  \left( \prod_{j=0}^p \frac{4\binom{d}{p}^2}{\tau _j}\right)^2
  = \frac{d\, \left(2\binom{d}{p}\right)^{4(p+1)}}{4\, \prod _{j=0}^p\tau _j^2}.
\end{equation}

\bigskip

For future reference we record the partial derivatives of
$\varphi$ here.
From (\ref{eq:A-def}) we calculate
\[\frac{\partial A}{\partial z}=-\frac{\partial B}{\partial z}=d-4p
\quad\text{and}\quad
\frac{\partial A}{\partial z_{j\alpha \beta }}=-\frac{\partial B}{\partial z_{j\alpha \beta }}
= (-1)^{\alpha +\beta } j\]
for \(j=1,\dots,p\).
Next, by (\ref{eqn:implicit-z0ab}),
if \(j\ge 1\) and \(\alpha ,\beta ,\gamma ,\delta \in \{0,1\}\)
such that \((\alpha ,\beta )\ne (\gamma ,\delta )\), we have
\[
\frac{\partial z_{0\alpha \beta }}{\partial z_{j\alpha \beta }}=-1,
\quad \frac{\partial z_{0\alpha \beta }}{\partial z_{j\gamma \delta }}=0, \,\,\text{ and }\,\,
\frac{\partial z_{0\alpha \beta }}{\partial z}=(-1)^{\alpha +\beta }.
\]
Hence
\begin{align}  
  \frac{\partial \varphi }{\partial z}
  &=(d-4p)(\log A-\log B) + \log\left(\frac{z_{010}z_{001}}{z_{011}z_{000}}\right),
  \label{eqn:partial-phi-partial-z} 
\end{align}
and for $(j,\alpha,\beta)\in [p]\times\{0,1\}^2$,
\begin{align}
  \frac{\partial \varphi }{\partial z_{j\alpha \beta }}
  &= j(-1)^{\alpha +\beta }(\log A - \log B)
  +\log\left(\frac{\tau _j}{\tau _0}\right)+\log(z_{0\alpha \beta })-\log(z_{j\alpha \beta }).
  \label{eqn:partial-phi-partial-zjab} 
\end{align}

\subsection{Applying Laplace's Summation Method}\label{ss:laplace}

In this section we apply the method of Laplace summation to the sum
(\ref{eqn:y2-asympt}).  We use the following
formulation~\cite{GREENHILL-JANSON-RUCINSKI-2010}.

\begin{lemma}
  \label{lemma:laplace-summation}
  Suppose the following:
  \begin{enumerate}
  \item[\emph{(i)}] \(\mathcal{L} \subset \mathbb{R} ^r\) is a lattice of full rank \(r\).
  \item[\emph{(ii)}] \(K\subset \mathbb{R} ^r\) is a compact convex set with non-empty interior.
  \item[\emph{(iii)}] \(\varphi : K\rightarrow \mathbb{R} \) is a continuous function with a unique maximum at some interior
    point \(\boldsymbol{x} _0\) of \(K\).
  \item[\emph{(iv)}] \(\varphi \) is twice continuously differentiable in a neighbourhood of \(\boldsymbol{x} _0\)
    and the Hessian \(D^2\varphi (\boldsymbol{x}_0)\) is 
    negative definite.
  \item[\emph{(v)}] \(\psi :K_1\rightarrow \mathbb{R} \) is a continuous function on some neighbourhood \(K_1\subseteq K\) of
    \(\boldsymbol{x}_0\) with \(\psi (\boldsymbol{x} _0)>0\).

  \item[\emph{(vi)}] For each positive integer \(n\) there is a vector \(\ell _n\in \mathbb{R} ^r\).

  \item[\emph{(vii)}] For each positive integer \(n\) there is a positive real number
    \(b_n\) and a function \(a_n:(\mathcal{L} +\ell _n)\cap nK\) such that, as \(n\rightarrow \infty \)
    \[a_n(\ell )=O\left(b_ne^{n\varphi (\ell /n)+o(n)}\right)\]
    uniformly for all \(\ell \in (\mathcal{L} +\ell _n)\cap nK\), and
    \[a_n(\ell )=b_n\left(\psi (\ell /n)+o(1)\right)e^{n\varphi (\ell /n)}\]
    uniformly for all \(\ell \in (\mathcal{L} +\ell _n)\cap nK_1\).
  \end{enumerate}
  Then, as \(n\rightarrow \infty \),
  \[\sum _{\boldsymbol{k} \in (\mathcal{L} +\ell _n)\cap nK}a_n(\boldsymbol{k} ) \sim  \frac{(2\pi )^{r/2}\psi (\boldsymbol{x} _0)}{\Det(\mathcal{L} )\sqrt{\Det(-H)}}\, b_n\, n^{r/2}\, e^{n\varphi (\boldsymbol{x} _0)}.\]
\end{lemma}

To apply this lemma we need the determinant of $-H$, where $H:=D^2\varphi(\boldsymbol{z}^*)$ is the Hessian matrix of $\varphi$ evaluated at $\boldsymbol{z}^*$.
The proof of Lemma \ref{lemma:det-is-good} is presented in  Section~\ref{ss:Hessian}.

\begin{lemma}
  \label{lemma:det-is-good}
  If \(d\ge 3\) and \(1\leq p<d\) are integers, then
  \begin{equation}
    \label{eqn:det-is-good}
  \Det(-H) =  \left(2\binom{d}{p}\right)^{8(p+1)}\,
 \frac{d^2(d-1)-(d^2-4dp+4p^2-d)^2}{32 d^2(d-1)\, \prod _{j=0}^p\tau _j^4}.
  \end{equation}
\end{lemma}

We remark that Lemma~\ref{lemma:det-is-good} holds for any \(p\) and \(d\), regardless
of condition (\ref{eq:inequality}). However, if (\ref{eq:inequality}) fails
then the determinant \(\Det(-H)\) is negative.
So this is another point in the argument
where assumption (\ref{eq:inequality}) is required in order
for the SSCM to apply to $Y$. 

\bigskip

Now we are able to complete the computation of the second moment for the pairs $(d,p)$ listed in (\ref{eq:list}).  The same proof would work for any pair $(d,p)$ which satisfies the conclusion of Lemma~\ref{eqn:hard-phi-maximum-lemma}.

\begin{lemma}
  \label{lem:2nd-mom}
  If  $(d,p)$ is one of the pairs listed in \emph{\eqref{eq:list}}
  then
  \[\frac{\Exp (Y^2)}{\Exp (Y)^2}\sim d\sqrt{\frac{d-1}{d^2(d-1)-(4p^2-4dp+d^2-d)^2}}.\]
\end{lemma}

\begin{proof}
For Lemma \ref{lemma:laplace-summation},
take \(\varphi \), \(\psi \), \(b_n\), and \(K\) as defined in Section \ref{ss:Ey2-asympt-sum}
and \(b_n=J_n\). Further, let \(\ell _n=\boldsymbol{0}\) be the zero vector.
Also, let \(\mathcal{L} =\mathbb{Z} ^{4p+1}\) and \(r=4p+1\).
Firstly, observe that \(I=nK\cap \mathcal{L} \), thus, the sum (\ref{eqn:y2-asympt}) has the required form.
We check each of the conditions of Lemma \ref{lemma:laplace-summation} in turn:
\begin{enumerate}[(i)]
\item The lattice \(\mathcal{L} \) has full rank \(4p+1\).

\item By definition, \(K\) is convex and compact with non-empty interior (see Section~\ref{ss:Ey2-asympt-sum}).

\item \(\varphi :K\rightarrow \mathbb{R} \) is continuous,
  and Lemma \ref{eqn:hard-phi-maximum-lemma} shows that \(\varphi \) has a unique maximum at the point
  \(\boldsymbol{z} ^{\ast}\), which lies in the interior of \(K\). Take \(\boldsymbol{x} _0=\boldsymbol{z} ^{\ast}\).

\item Equation (\ref{eqn:phi-def}) shows that \(\varphi \) is clearly twice continuously differentiable on the interior
  of \(K\), in fact, it is infinitely differentiable.
  The fact that \(\boldsymbol{z} ^{\ast}\) is a global maximum of $\varphi$ on the interior of \(K\) shows that the Hessian
   \(H\) is
  negative definite.

\item Equation (\ref{eqn:psi-def}) shows that \(\psi \) is positive and continuous on the interior of \(K\). Therefore, we can take
  \(K_1\subset K\) to be any open neighbourhood of \(\boldsymbol{z} ^{\ast}\).

\item We have already seen that \(\ell _n=\boldsymbol{0}\) is an appropriate choice.

\item This follows from Lemma~\ref{lem:J-asymptotic}.
\end{enumerate}
Thus, Lemma~\ref{lemma:laplace-summation} applies, and we conclude that
\[
\Exp (Y^2)\sim \frac{(2\pi )^{d/2}\,\psi (\boldsymbol{z} ^{\ast})}{\Det(\mathcal{L} )\, \sqrt{\Det(-H)}}\,b_n\, n^{d/2}\,e^{n\varphi (\boldsymbol{z} ^{\ast})}
\sim \frac{\left(2\binom{d}{p}\right)^{4(p+1)}}{\sqrt{32\, \Det(-H)}\prod _{j=0}^p\tau _j^2}\,\, (\Exp Y)^2, \label{eq:halfway}
\]
using (\ref{eqn:bn-def}),  (\ref{eqn:phi(zstar)}), (\ref{eqn:psi-zstar}),
and the fact that \(\Det(\mathcal{L} )=1\).
Substituting the result of Lemma~\ref{lemma:det-is-good}
completes the proof.
\end{proof}

\bigskip

We complete this section by proving  Theorem~\ref{thm:small-dp}(ii).

\begin{proof}[Proof of Theorem~\ref{thm:small-dp}(ii)]\
By Lemma~\ref{lem:list}, it suffices to prove the result
for the five pairs $(d,p)$ listed in (\ref{eq:list}). 
First we prove that Theorem~\ref{thm:sscm} can be applied to the random variables $Y$ and $X_i$, namely, the number of $p$-orientations and the number of $i$-cycles in
 \(F\in \mathcal{P} _{n,d}\), respectively.  It is well-known~\cite{BOLLOBAS1980311} that
 the $X_i$ are asymptotically independent Poisson random variables
 with $\Exp(X_i)=\lambda_i$, where $\lambda_i$ is defined in 
 (\ref{eq:lambda-def}).  Lemma~\ref{lemma:joint-moments} proves
 that assumption (A2) holds, Lemma \ref{lemma:expsum-lemma}
 proves that (A3) holds and
Lemma~\ref{lem:2nd-mom} proves that (A4) holds.
Therefore, Theorem \ref{thm:sscm} applies and proves that a.a.s.\ $\Pr(Y>0)$, 
noting from Lemma~\ref{lemma:joint-moments}
that $\delta_i > -1$ for all $i\geq 1$. 

Therefore $\Pr(Y=0)=o(1)$ and applying (\ref{eqn:pndsimple}) we conclude
that $\Pr(Y_\mathcal{G}=0)=o(1)$ for the corresponding random
variable $Y_{\mathcal{G}}$ defined over $\mathcal{G}_{n,d}$.
In other words, a.a.s.\ $\mathcal{G}_{n,d}$ has a 
$p$-orientation whenever $(d,p)$ is one of the pairs from
\eqref{eq:list}.  This completes the proof of Theorem~\ref{thm:small-dp}, using Lemma~\ref{lem:list}.
 \end{proof}

\subsection{The Hessian Determinant}\label{ss:Hessian}

To prove Theorem~\ref{thm:small-dp} we only need to know the result of Lemma~\ref{lemma:det-is-good} for the pairs $(d,p)$ listed in \eqref{eq:list}. However, we prefer to prove the result in general, to assist anyone who might like to extend our argument to other pairs outside Table~\ref{t:results}.  

To prove Lemma \ref{lemma:det-is-good} we find an explicit
expression for the Hessian matrix \(H\). Then we write \(H\)
as a sum of diagonal matrices
and products of simpler matrices and apply
various linear algebraic identities to compute
\(\Det(H)\) explicitly.
This extends a result from~\cite[Section 5.1]{stars}.
We will use the following three results.

\begin{lemma}[Generalised Matrix Determinant Lemma~{\cite[Exercise~6.2.7]{meyer00}}]
  \label{lemma:mdt}
  Let \(P\) be an invertible \(n\times n\) matrix and \(Q,R\) be \(n\times m\) matrices, then,
  \[\Det(P+QR^T)=\Det(I_m+R^TP^{-1}Q) \, \Det(P).\]
\end{lemma}

\begin{lemma}[Sherman-Morrison Formula {\cite[equation (3.8.2)]{meyer00}}]
  \label{lemma:smf}
  If \(P\) is an invertible matrix and \(\boldsymbol{u} \) and \(\boldsymbol{v} \) are vectors, then
  \[(P+\boldsymbol{u} \boldsymbol{v} ^T)^{-1}=P^{-1}-\frac{P^{-1}\boldsymbol{u} \boldsymbol{v} ^TP^{-1}}{1+\boldsymbol{v} ^TP^{-1}\boldsymbol{u} }\]
  provided that \(1+\boldsymbol{v} ^TP^{-1}\boldsymbol{u} \ne 0\).
\end{lemma}

The following lemma generalises Delcourt et al.~\cite[Lemma 5.4]{stars},
which corresponds to taking \(\lambda =-\mu =1\) and letting \(V\) be a vector.

\begin{lemma}
  \label{lemma:gross-det-formula}
  Suppose that \(H\) is an invertible real matrix of the form
  \[H=D+\lambda \boldsymbol{u} \boldsymbol{u} ^T+\mu VV^T\]
  for some scalars \(\lambda ,\mu \in \mathbb{R} \), an invertible matrix \(D\), a vector \(\boldsymbol{u} \),
  and a real matrix \(V\) (which need not be square). 
If \(1+\lambda \boldsymbol{u} ^TD^{-1}\boldsymbol{u} \ne 0\) then
  \begin{align*}
    \Det(H)=\left(1+\lambda \boldsymbol{u} ^TD^{-1}\boldsymbol{u} \right)\Det(D)
    \Det\left(
    I+\mu V^TD^{-1}V-\frac{\lambda \mu V^TD^{-1}\boldsymbol{u} \boldsymbol{u} ^TD^{-1}V}{1+\lambda \boldsymbol{u} ^TD^{-1}\boldsymbol{u} }
    \right).
  \end{align*}
\end{lemma}

\begin{proof}
  Applying Lemma \ref{lemma:mdt}) twice, we obtain
  \begin{align}
    \Det(H)&= \Det(D+\lambda \boldsymbol{u} \boldsymbol{u} ^T)\,\Det\bigg(I+\mu V^T\left(D+\lambda \boldsymbol{u} \boldsymbol{u} ^T\right)^{-1}V\bigg) \notag\\
&= (1+\lambda \boldsymbol{u} ^TD^{-1}\boldsymbol{u} )\, \Det(D)\,
\Det\bigg(I+\mu V^T\left(D+\lambda \boldsymbol{u} \boldsymbol{u} ^T\right)^{-1}V\bigg).
        \label{eqn:det-lemma-easy}
  \end{align}
  Next, applying the Sherman-Morrison formula (Lemma \ref{lemma:smf}), we obtain
  \[(D+\lambda \boldsymbol{u} \boldsymbol{u} ^T)^{-1}=D^{-1}-\frac{\lambda D^{-1}\boldsymbol{u} \boldsymbol{u} ^TD^{-1}}{1+\lambda \boldsymbol{u} ^TD^{-1}\boldsymbol{u} },\]
  using the assumption that \(1+\lambda \boldsymbol{u} ^TD^{-1}\boldsymbol{u} \ne 0\).
 The proof is completed by substituting this into (\ref{eqn:det-lemma-easy}) and rearranging.
\end{proof}

Next, we find an expression the Hessian matrix.
We label the rows and columns of our matrices by the variables in $\boldsymbol{z}$, with the first
row and column corresponding to $z$ and all other rows and columns ordered in lexicographical order. That is, rows 2--5 correspond to $z_{100}$, $z_{101}$, $z_{110}$, $z_{111}$ in this order, the next four rows to $z_{200},z_{201}, z_{210}, z_{211}$ in this order, and so on; and similarly for columns.

\begin{lemma}
Suppose that $(d,p)$ is one of the pairs from \eqref{eq:list}).
Let $E\in M_{4p+1}(\mathbb{R})$ be the $(4p+1)\times (4p+1)$ matrix
with a 1 in the (1,1) entry and zeros everywhere else. Let $\boldsymbol{u} = \big(u, (u_{j00},u_{j01},u_{j10},u_{j11})_{j\in [p]}\big)\in\mathbb{R}^{4p+1}$
be defined by
\[ u\coloneq d-4p \quad\text{and}\quad u_{j\alpha \beta }\coloneq  (-1)^{\alpha +\beta }j.\]
Finally, let $V\in M_{4p+1,4}(\mathbb{R})$ be the $(4p+1)\times 4$ matrix defined by
\[V\coloneq 
\begin{pmatrix}
  \boldsymbol{v}  & I_4 & I_4 & \cdots & I_4
\end{pmatrix}^T,
\]
where the identity matrix appears \(p\) times and \(\boldsymbol{v} :=(-1,1,1,-1)^T\).
Then the Hessian matrix of $\varphi$ at the point $\boldsymbol{z}$ is given by
\begin{equation}
  \label{eqn:hessian-decomposition}
  H=D-E +\frac{8}{d}\boldsymbol{u} \boldsymbol{u} ^T - \frac{1}{z_{000}^{\ast}}VV^T\ .
\end{equation}
\end{lemma}

\begin{proof}

Differentiating equations (\ref{eqn:partial-phi-partial-z}) and (\ref{eqn:partial-phi-partial-zjab}), we obtain
the following:
\begin{align*}
  \frac{\partial ^2\varphi }{\partial z^2}
  &=(d-4p)^2\left(\frac1A+\frac1B\right)-\left(\frac{1}{z_{000}}+\frac{1}{z_{001}}+\frac{1}{z_{010}}+\frac{1}{z_{011}}\right), \\
  \frac{\partial ^2\varphi }{\partial z\partial z_{j\alpha \beta }}
  &=(-1)^{\alpha +\beta }\left(\frac{j(d-4p)}A+\frac{j(d-4p)}B+\frac{1}{z_{0\alpha \beta }}\right), \\
   \frac{\partial ^2\varphi }{\partial z_{j\alpha \beta }\partial z_{k\gamma \delta }}
  &= 
  jk(-1)^{\alpha +\beta +\gamma +\delta }\left(\frac1A+\frac1B\right) -
  \frac{\mathbb{1}\big( (\alpha,\beta)=(\gamma,\delta)\big)}{z_{0\alpha\beta}}
  -  \frac{\mathbb{1}\big( (j,\alpha,\beta)=(k,\gamma,\delta)\big)}{z_{j\alpha\beta}}
\end{align*}
for all $(j,\alpha,\beta),(k,\gamma,\delta)\in [p]\times\{0,1\}^2$.
Here $\mathbb{1}(\mathcal{A})$ is an indicator function which equals 1 if the event $\mathcal{A}$ is true, and 0 otherwise.
Substituting \(\boldsymbol{z} =\boldsymbol{z} ^{\ast}\) using (\ref{eqn:equation-summat}),
recalling that $A=B=d/4$ at $\boldsymbol{z}^\ast$, 
we find that the Hessian at $\boldsymbol{z}^*$ has entries
\begin{align*}
  \frac{\partial ^2\varphi }{\partial z^2}
  &= \frac8d(d-4p)^2-\frac{4}{z_{000}^{\ast}}, \\
  \frac{\partial ^2\varphi }{\partial z\partial z_{j\alpha \beta }}
  &= \frac8d(-1)^{\alpha +\beta }j(d-4p)+\frac{(-1)^{\alpha +\beta }}{z_{000}^{\ast}}, \\
\frac{\partial ^2\varphi }{\partial z_{j\alpha \beta }\partial z_{k\gamma \delta }} &=
\frac8d (-1)^{\alpha +\beta +\gamma +\delta }jk
 -
  \frac{\mathbb{1}\big( (\alpha,\beta)=(\gamma,\delta)\big)}{z_{000}^\ast}
  -  \frac{\mathbb{1}\big( (j,\alpha,\beta)=(k,\gamma,\delta)\big)}{z_{j00}^\ast}.
\end{align*}
Observe that the product of \(V\) and \(V^T\) is the
\((4p+1)\times (4p+1)\) matrix of the form
\[VV^T=
\begin{pmatrix}
  4 & \boldsymbol{v} ^T & \boldsymbol{v} ^T & \cdots & \boldsymbol{v} ^T \\ 
  \boldsymbol{v}  & I_4 & I_4 & \cdots & I_4 \\
  \vdots & \vdots & \vdots & \ddots & \vdots \\
  \boldsymbol{v}  & I_4 & I_4 & \cdots & I_4
\end{pmatrix}.
\]
Note that the entry in the row indexed by $(j,\alpha,\beta)$ and the column indexed by $(k,\gamma,\delta)$
is given by $\mathbb{1}\big((\alpha,\beta)=(\gamma,\delta)\big)$.
The result follows.
\end{proof}

We cannot apply Lemma~\ref{lemma:gross-det-formula} directly to equation (\ref{eqn:hessian-decomposition})
since the matrix \(D-E\) has a zero in the $(1,1)$ entry and hence is not invertible.
Instead, since the determinant is a linear function of the first row of its input, we have  
\begin{equation}
    \Det(H) = \Det(J_1) - \Det(J_0)
    \label{eqn:det-and-minor-det}
\end{equation}
where
\[ J_1:= D+\frac8d\boldsymbol{u} \boldsymbol{u} ^T-\frac{1}{z_{000}^{\ast}}VV^T
\]
and  $J_0$ is the $(1,1)$-minor of $J_1$, obtained from $J_1$ by deleting the first row
and first column.  
Similarly, we make the following definitions:
\begin{itemize}
\item 
Let $D_1:= D$ and let $D_0$ be the $(1,1)$-minor of $D_1$. Hence
$D_0\in M_{4p}(\mathbb{R})$ is the diagonal matrix with
diagonal elements \((d_{j00},d_{j01},d_{j10},d_{j11})_{j\in [p]}\).  Furthermore, $\Det(D_\chi) = \Det(D)$
for $\chi\in\{0,1\}$, since the $(1,1)$ entry of $D$ is 1.
\item Let $V_1:= V$ and define $V_0$ by deleting the first row of $V_1$. Then
\[V_0 = 
\begin{pmatrix}I_4&I_4&\dots&I_4\end{pmatrix}^T
\]
and \(V_0V_0^T\) is the $(1,1)$-minor of $V_1V_1^T$. 
\item
Finally, let \(\boldsymbol{u}_1:=\boldsymbol{u}\) and let $\boldsymbol{u}_0$ be obtained
from $\boldsymbol{u}_1$ by deleting the first entry.  That is,  
\(\boldsymbol{u}_0 =(u_{j00},u_{j01},u_{j10},u_{j11})_{j\in [p]}\). Then $\boldsymbol{u}_0\boldsymbol{u}_0^T$ is the $(1,1)$-minor of $\boldsymbol{u}_1\boldsymbol{u}_1^T$.
\end{itemize}
These definitions allow us to write $J_1$ and $J_0$ in a common form: for $\chi\in\{0,1\}$ we have
\[ J_\chi =D_\chi +\frac8d\boldsymbol{u}_\chi\boldsymbol{u}_\chi^T-\frac{1}{z_{000}^{\ast}}V_\chi V_\chi^T.\]
Next we apply Lemma \ref{lemma:gross-det-formula} to calculate the determinants of these two matrices.

\begin{lemma}
\label{lemma:H11-det}
For $\chi\in\{0,1\}$ we have
  \begin{align*}
    \Det(J_\chi)
     &=   \frac{\Det(D)}{ \big(4z_{000}^{\ast}\big)^4}\,
    \left( (1-16\chi)\left(1+\dfrac{8}{d}\, \boldsymbol{u}_{\chi}^T D_{\chi}^{-1}\boldsymbol{u}_{\chi} \right) + \frac{128}{d}\Big( (d-4p)\chi + p^2/(4d)\Big)^2\right).
  \end{align*}
\end{lemma}

\begin{proof}
  Recalling that $\Det(D_\chi) = \Det(D)$ for $\chi\in\{0,1\}$, applying Lemma \ref{lemma:gross-det-formula} to $J_\chi$ gives
  \begin{align}
    \Det(J_{\chi})
    &=\left(1+\frac8d\boldsymbol{u}_\chi^T D_\chi^{-1} \boldsymbol{u}_\chi\right)\, \Det(D)\,
    \Det(P_{\chi})
    \label{eq:detJ-intermediate}
    \end{align}
    where $P_\chi$ is the $4\times 4$ matrix defined by
    \[ P_\chi:= 
    I_4-\frac{V_\chi^T D_\chi^{-1} V_\chi}{z_{000}^{\ast}}+\frac{8V_\chi^T D_\chi^{-1}\boldsymbol{u} _\chi\boldsymbol{u} _\chi^T D_\chi^{-1} V_\chi}
      {z_{000}^{\ast}(d+8\boldsymbol{u}_\chi^T D_\chi^{-1} \boldsymbol{u}_\chi)}.
\]
Observe that
  \[
  V_\chi^T D_\chi^{-1} V_\chi = \chi \boldsymbol{v}\boldsymbol{v}^T -\left(\sum _{j\in [p]} z_{j00}^{\ast}\right)I_4
   = \chi \boldsymbol{v}\boldsymbol{v}^T -\Big(\dfrac{1}{4} - z_{000}^{\ast}\Big)I_4,
  \]
  using (\ref{eqn:equation-summat}) and Lemma~\ref{lemma:tau-sums}.
  Next, we compute that
  \[ 
  V_\chi^T D_\chi^{-1} \boldsymbol{u} _\chi\boldsymbol{u}_\chi^T D_\chi^{-1} V_\chi =
  \left( \chi(d-4p) + \sum _{j\in [p]} jz_{j00}^{\ast}\right)^2\boldsymbol{v} \boldsymbol{v} ^T
  =  \left( \chi(d-4p) + \frac{p^2}{4d}\right)^2\boldsymbol{v} \boldsymbol{v} ^T
  \]
  again using (\ref{eqn:equation-summat}) and Lemma~\ref{lemma:tau-sums}.
  Consequently, 
\[ P_\chi = \frac{1}{4z_{000}^\ast}\, I_4 + \left( -\frac{\chi}{z_{000}^\ast} +
\frac{8\big(\chi(d-4p) + p^2/(4d)\big)}{z_{000}^\ast\, \big(d + 8\boldsymbol{u}_{\chi}^T D_{\chi}^{-1} \boldsymbol{u}_\chi\big)}
\right)^2\, \boldsymbol{v}\boldsymbol{v}^T
\]   
  and we may apply Lemma~\ref{lemma:mdt}
  one more time to obtain
  \[
  \Det(P_\chi)=\left(1 - 16\chi +\frac{128\big( (d-4p)\chi + p^2/(4d)\big)^2}{d+8\boldsymbol{u}_\chi^T D_\chi^{-1} \boldsymbol{u}_\chi }\right)\,
  \frac{1}{\big(4z_{000}^{\ast}\big)^4}\,
  \]
  using the fact that $\boldsymbol{v}^T\boldsymbol{v}=4$.
Substituting this into (\ref{eq:detJ-intermediate}) completes the proof. 
\end{proof}

To complete the evaluation of $\Det(-H)$,  observe that
\begin{align*}
  \boldsymbol{u}_{1}^T D_{1}^{-1}\boldsymbol{u}_{1} &= 
  (d-4p)^2 +  \boldsymbol{u}_{0}^T D_{0}^{-1}\boldsymbol{u}_{0},\\
   \boldsymbol{u}_{0}^T D_{0}^{-1}\boldsymbol{u}_{0} &= -4\sum_{j\in [p]} j^2 z_{j00}^\ast =
  - \frac{p^2(p^2-2p+d)}{d(d-1)}
\end{align*}
using Lemma~\ref{lemma:tau-sums} and (\ref{eqn:equation-summat}).
Therefore, by Lemma~\ref{lemma:H11-det},
\begin{align}
& \frac{(4z_{000}^\ast)^4}{\Det(D)}\, \Det(H) \\&= \frac{(4z_{000}^\ast)^4}{\Det(D)}\,\big(\Det(J_1) - \Det(J_0)\big) \notag \\
  &= -15\Big(1 + \frac{8}{d} \boldsymbol{u}_1^T D_1^{-1} \boldsymbol{u}_1\Big) + \frac{128}{d}\left( d-4p + \frac{p^2}{4d}\right)^2 - \Big(1 + \frac{8}{d} \boldsymbol{u}_0^T D_0^{-1} \boldsymbol{u}_0\Big) - \frac{128}{d}\left(\frac{p^2}{4d}\right)^2 \notag \\
  &=  -16 - \frac{120(d-4p)^2}{d} - \frac{128}{d}  \boldsymbol{u}_{0}^T D_{0}^{-1}\boldsymbol{u}_{0}
    + \frac{128}{d}\left( (d-4p)^2 + \frac{p^2(d-4p)}{2d} \right) \notag \\
    &= - \frac{8\big( d^2(d-1) - (d^2-4dp+4p^2-d)^2\big)}{d^2(d-1)}.\label{eq:scaled-detH}
\end{align}
Finally, observe that
\begin{equation}
    \label{eq:detD-factor}
\frac{\Det(D)}{(4z_{000}^\ast)^4} = \left( 4\prod_{j=0}^p z_{j00}^\ast\right)^{-4} = \frac{\left( 2\binom{d}{p}\right)^{8(p+1)}}{256\prod_{j=0}^p \tau_j^4}
\end{equation} 
by definition of $D$ and using (\ref{eqn:equation-summat}).
Since $\Det(-H) = (-1)^{4p+1}\, \Det(H) = -\Det(H)$, the proof of Lemma~\ref{lemma:det-is-good} 
follows from combining (\ref{eq:scaled-detH}) and (\ref{eq:detD-factor}).



\appendix

\section{Second moment calculations}\label{s:hard}

The goal of this appendix is to prove Lemma \ref{eqn:hard-phi-maximum-lemma}. 
Since \(\varphi \) is differentiable on the interior of \(K\), it is enough to show that
\begin{enumerate}[(I)]
\item The only critical point of \(\varphi \) on the interior of \(K\) is \(\boldsymbol{z} ^{\ast}\); and
\item For every boundary point \(\boldsymbol{z} \) of \(K\) we have \(\varphi (\boldsymbol{z} )<\varphi (\boldsymbol{z} ^{\ast})\).
\end{enumerate}

\subsection{The interior of $K$}\label{ss:interior-of-K}

Setting all partial derivatives to zero, for any \(1\le j\le p\)
and \(\alpha ,\beta \in \{0,1\}\), equations (\ref{eqn:partial-phi-partial-z})
and (\ref{eqn:partial-phi-partial-zjab})
imply that
\[
  0=\frac{\partial \varphi }{\partial z_{j\alpha \beta }}-\frac{j(-1)^{\alpha +\beta }}{d-4p}\frac{\partial \varphi }{\partial z}
  =
  \log\left[\frac{\tau _jz_{0\alpha \beta }}{\tau _0z_{j\alpha \beta }}\left(\frac{z_{011}{z_{000}}}{z_{010}{z_{001}}}\right)^{\frac{j(-1)^{\alpha +\beta }}{d-4p}}\right].
\]
Rearranging this equation and considering all choices of \(\alpha \) and \(\beta \), we obtain the relation
\begin{equation}
  \label{eqn:j-depends-on-0}
  z_{j\alpha\beta}= \frac{\tau_j}{\tau_0} \, z_{0\alpha\beta}\, R^{(-1)^{\alpha+\beta}j}
\end{equation}
where
\begin{equation}
  \label{eqn:Rdef}
  R=R(z_{000},z_{001},z_{010},z_{011})\coloneq \left(\frac{z_{000}z_{011}}{z_{001}z_{010}}\right)^{\frac{1}{d-4p}}.
\end{equation}
Since \(\boldsymbol{z} \) is an interior point we have 
 \(R>0\).
The system (\ref{eqn:j-depends-on-0}) shows that the variables
\(z_{j\alpha \beta }\) for $j\in [p]$ and \(\alpha ,\beta \in \{0,1\}\) can be parameterised at the critical
points of \(\varphi \) by
five parameters \(z, z_{000},z_{001},z_{010},z_{011}\).

Let \(y=A/B\). Then, equation (\ref{eqn:partial-phi-partial-z}) implies that at a critical point,
\[\frac{z_{000}z_{011}}{z_{001}z_{010}}=y^{d-4p},\]
that is, \(y=R\). 
Next, by summing (\ref{eqn:j-depends-on-0}) and using (\ref{eqn:implicit-z0ab}) we obtain
\begin{equation}
  \label{eqn:z000,011-param}
  z_{000}=z_{011}=\frac{\tau _0z}{\sum _{j=0}^p\tau _jy^j}\quad\text{and}\quad
  z_{001}=z_{010}=\frac{(1-2z)\tau _0}{2\sum _{j=0}^p\tau _j(1/y)^j}.
\end{equation}
This gives a parametrisation of the critical points in terms
of only \(y\) and \(z\).
For brevity, we define the following polynomials:
\begin{align}
  \theta (t)\coloneq \sum _{j=0}^p\tau _jt^j, \quad
  &\theta_1(t)\coloneq t\theta '(t)=\sum _{j\in [p]} j\tau _jt^j,  \label{eqn:theta-def}\\
  \eta(t) \coloneq t^p \theta(1/t) = \sum_{j=0}^p \sigma_j t^j,\quad
  &\eta_1(t)\coloneq t^p \theta_1(1/t) = \sum_{j=0}^{p-1} (p-j) \sigma_j t^j. \label{eqn:eta-def}
\end{align}
Note that \(\eta\) is the reciprocal polynomial of \(\theta\) and \(\eta_1\) is the reciprocal
polynomial of \(\theta_1\).
Combining equation (\ref{eqn:z000,011-param}) 
above
with equation (\ref{eqn:j-depends-on-0}), we obtain
\begin{equation}
  \label{eqn:zjab-theta-param}
  z_{j00}=z_{j11}=\frac{\tau _jzy^j}{\theta (y)}\quad\text{and}\quad
  z_{j01}=z_{j10}=\frac{\tau _j(1-2z)(1/y)^j}{2\theta (1/y)}.
\end{equation}
for \(j=0,\dots,p\).
Allowing \(z\) to vary on the interval \([0,1/2]\) and \(y\) to vary
on the interval \((0,\infty )\) independently of \(z\) provides a parameterisation of critical points of \(\varphi \) on the interior of
\(K\). 
In fact, every point of \(K\) at which the relation
\(A/B=R=y\)
is satisfied is, by construction, a critical point of \(\varphi \) on the interior of \(K\),
and conversely every critical point satisfies this relation.

Pra{\l}at and Wormald~\cite{pralat15} and
Delcourt et al.~\cite{delcourt25}
used resultants (see for example~\cite[pg.~102-107]{vdw}) to reduce the problem of solving a system of
polynomial equations such as (\ref{eqn:zjab-theta-param})
to the problem of classifying the roots of a specific polynomial.  
We avoid the use of resultants
by eliminating variables from the system of polynomial equations more efficiently. 

We now
reformulate the relation \(y=A/B=R\).
Recall that \(A+B=d/2\), hence, the relation \(y=A/B\) is equivalent to
\[A=\frac{dy}{2(1+y)}.\]
Applying (\ref{eq:A-def}) we require that
\[
\frac{dy}{2(1+y)}=p+(d-4p)z+\sum _{j\in [p]} j(z_{j00}+z_{j11}-z_{j01}-z_{j10}).
\]
Thus, (\ref{eqn:zjab-theta-param}) yields the equivalent form
\begin{align}
  \frac{dy}{2(1+y)}&=
  \label{eqn:first-big-zed}
  p-\frac{\eta _1(y)}{\eta (y)}
  + z\left[d-4p+\frac{2\theta _1(y)}{\theta (y)}+\frac{2\eta _1(y)}{\eta (y)}\right].
\end{align}

Next observe that since \(y>0\) and \(R>0\), equations
(\ref{eqn:Rdef}) and
(\ref{eqn:zjab-theta-param}) imply that the relation $y=R$ is equivalent to
\[
y^{d-4p}=\left[\frac{\tau _0z/\theta (y)}{\tau _0(1-2z)/(2\theta (1/y))}\right]^2
=
\left[\frac{2z\theta (1/y)}{(1-2z)\theta (y)}\right]^2.
\]
Since \(y>0\) and 
\(0<z<1/2\), we have 
\[\frac{2z\theta (1/y)}{(1-2z)\theta (y)} > 0.\]
Taking square roots,
we conclude that the relation \(y=R\) is equivalent to the relation
\begin{equation}
  \label{eqn:second-big-zed}
  w^{d-4p}=\frac{2z\theta (1/w^2)}{(1-2z)\theta (w^2)},
\end{equation}
where \(w=\sqrt{y}\).
Solving equations (\ref{eqn:first-big-zed}) and (\ref{eqn:second-big-zed}) simultaneously,
we see that the relation
\(y=A/B=R\) is equivalent to the polynomial equation
\(f(w)=0,\)
where \(w>0\) and
\begin{align}
  f(w)&\coloneq  w^{d-2p}f_1(w^2) + f_2(w^2), \label{eqn:expandf} \\
  f_1(y)&\coloneq  (-4y-4)\theta _1(y) + \left(4py+4p-2d\right)\theta (y), \nonumber\\
  f_2(y)&\coloneq  (4y+4)\eta _1(y)+\left(2dy-4py-4p\right)\eta (y). \nonumber
\end{align}
It follows, therefore, that a full classification of the positive roots of \(f\) yields a full
classification of the critical points of \(\varphi \) on the interior of \(K\). In particular, we will
prove that if the assumptions of Theorem \ref{thm:small-dp} then
\(f\) has a unique positive real root at \(w=1\).
It is easy to confirm that \(w=1\) is a solution, indeed,
\[f_1(1)=-8\theta_1(1)+(8p-2d)\theta(1)=-8\eta_1(1)+(8p-2d)\eta(1)=-f_2(1),\]
and it follows immediately that \(f(1)=0\).
When $w=1$ we also have $y=1$ and, by equation (\ref{eqn:second-big-zed}),
\[z=\frac{\theta (1)}{4\theta (1)}=\frac14.\]
Let \(z^{\ast}=1/4\) be this value of \(z\) and denote by \(z_{j\alpha \beta }^{\ast}\) the value of \(z_{j\alpha \beta }\) at this critical point
for \(\alpha ,\beta \in \{0,1\}\)
and \(0\le j\le p\). Then (\ref{eqn:zjab-theta-param})
implies that
\[z_{j\alpha \beta }^{\ast}=\frac{\tau _j}{4\theta (1)}=\frac{\tau _j}{4\binom dp^2},\]
by Lemma \ref{lemma:tau-sums}.
This shows that \(\boldsymbol{z} ^{\ast}\) is a critical point of \(\varphi \), and, by definition,
\(\boldsymbol{z} ^{\ast}\) lies in the interior of \(K\).
All that remains is to prove uniqueness.

Suppose that \(w_0\) is any root on the interval \((0,1)\) and observe that
\(-f_1\) is the reciprocal polynomial
of \(f_2\).
Thus,
\[f(1/w_0)=\frac{-1}{w_0^{d-2p}}w_0^{d-2p}f_2(w_0^2)-f_1(1/w_0^2)=-f(w_0)=0,\]
that is, \(1/w_0\) is a root of \(f\) on the interval \((1,\infty )\).
This proves the following criterion.
\begin{lemma}
  \label{lemma:polynomial-criterion}
  The function \(\varphi\) has a unique critical point on the interior of \(K\) if and only
  if the polynomial \(f\) has no root on the interval \((1,\infty)\).
\end{lemma}
It is straightforward to verify this criterion in any particular case from \eqref{eq:list}. For instance,
if the Taylor expansion of \(f\) centred at \(w=1\) has all nonnegative coefficients, then
clearly \(f\) has no root on \((1,\infty)\).
For example, let \(x=w-1\) and suppose that $(d,p)=(13,3)$: here
\begin{align*}
  f(w)& =
  26884 {{x}^{13}}+349492 {{x}^{12}}+2217072 {{x}^{11}}+9010144 {{x}^{10}} \\
  &\qquad\qquad+25545520 {{x}^{9}}+52351728 {{x}^{8}}+78942864 {{x}^{7}}+88217844 {{x}^{6}} \\
  &\qquad\qquad+ 72935148 {{x}^{5}}+44129800 {{x}^{4}}+18997264 {{x}^{3}}+5373368 {{x}^{2}}+767624
\end{align*}
which has only nonnegative coefficients.
Similar details for the other pairs $(d,p)$ from \eqref{eq:list} are presented in Appendix~\ref{ss:numerical-unique}.
Therefore, applying Lemma \ref{lemma:polynomial-criterion},
our analysis of the interior of \(K\) is complete.

\begin{lemma}
  If $(d,p)$ is one of the pairs from \emph{\eqref{eq:list}} then \(\varphi\) has a unique critical point
  on the interior of \(K\).
\end{lemma}

\subsection{Preliminary analysis of the boundary of $K$} \label{ss:Kint-directional-stuff}

In this section we show that \(\varphi \) has no maxima at a boundary point \(\boldsymbol{z} \) of \(K\)
such that \(z\ne 0\) and \(z\ne 1/2\).
Throughout this section we assume that \(z\in (0,1/2)\) is fixed and consider the cases of \(z=0,1/2\) in Sections~\ref{section:maxima-phi0} and
\ref{ss:boundaryko}.  
Let \(\boldsymbol{e} _{j\alpha \beta }\)
be the standard basis vector in the direction of \(z_{j\alpha \beta }\). 

The following simple inequalities inform our analysis. 

\begin{lemma} 
  \label{lemma:ABnonzero}
  For all \(\boldsymbol{z} \in K\) we have $A\ge (d-2p)z$
  with equality only if \(z_{p01}=z_{p10}=1/2\).  Moreover,
  $B\ge (d-2p)(1/2-z)$.
  Hence, if \(z\in (0,1/2)\) is fixed, then
  \(A\) and \(B\) are positive and bounded away from \(0\). 
\end{lemma}
\begin{proof}
  We prove the first inequality and note that the second follows by symmetry
  \big(cf.~equation (\ref{eqn:phisymm})\big).
  The identity (\ref{eqn:implicit-z0ab}) implies that
  \[\sum_{j\in [p]} j(z_{j10}+z_{j01}-z_{j00}-z_{j11})
  \le p\sum_{j=0}^p \left(z_{j10}+z_{j01}\right)
  =p-2pz.\]
  The inequality then follows from (\ref{eq:A-def}).
\end{proof} 

Let \(\boldsymbol{a} =\left(a,(a_{j00},a_{j01},a_{j10},a_{j11})_{j\in [p]}\right)\) be a point on the boundary of \(K\), where
\(a\ne 0,1/2\). 
Since \(K\) is a convex region, \(\boldsymbol{a} \) can be approached
along a line segment 
\[\boldsymbol{z} =\boldsymbol{z} (\epsilon )=\boldsymbol{a} +\varepsilon \boldsymbol{c} \]
contained in $K^{\circ}$,
where \(\boldsymbol{c} =\left(0,(c_{j00},c_{j01},c_{j10},c_{j11})_{j\in [p]}\right)\) is the direction of
the line segment and \(\varepsilon >0\) is sufficiently small.
Note that if \(a_{j\alpha \beta }=0\) then \(c_{j\alpha \beta }>0\). 

For convenience, given distinct \(\alpha ,\beta \in \{0,1\}\), define
\begin{equation}
  a_{0\alpha \alpha }=a-\sum _{j\in [p]} a_{j\alpha \alpha },\quad 
  a_{0\alpha \beta }=\dfrac12-a-\sum _{j\in [p]}a_{j\alpha \beta },\quad
  \text{and}\quad
  c_{0\alpha \beta }=-\sum _{j\in [p]} c_{j\alpha \beta }.
\end{equation}
Thus, if \(\sum _{j\in [p]} z_{j\alpha \alpha }\rightarrow z\) as \(\varepsilon \rightarrow 0^+\) then \(a_{0\alpha \alpha }=0\). Similarly, if
\(\alpha \ne \beta \) and
\(\sum _{j\in [p]} z_{j\alpha \beta }\rightarrow 1/2-z\) as $\varepsilon\to 0^+$ then \(a_{0\alpha \beta }=0\).
Thus, \(\boldsymbol{a} \) is a boundary point if and only
if \(a_{j\alpha \beta }=0\) for some \(0\le j\le p\) and some \(\alpha ,\beta \in \{0,1\}\).

The following lemma characterises the behaviour of \(\Delta \) near the boundary of \(K\).
\begin{lemma}
  \label{lemma:Kboundary}
  If \(\boldsymbol{a} \) is a point on the boundary of \(K\) with \(a\ne 0,1/2\) then $\boldsymbol{a}$ is not a local maximum of $\varphi$.

\end{lemma}

\begin{proof}
Let \(\boldsymbol{c} \) be defined as above. 
The directional derivative of \(\varphi \) of \(\boldsymbol{z} \) as \(\varepsilon \) approaches zero from above is directly proportional
to 
\[\Delta \coloneq -\nabla \cdot \varphi =-\sum _{j\in [p]} \sum _{\alpha ,\beta \in \{0,1\}} c_{j\alpha \beta }\, \frac{\partial \varphi }{\partial z_{j\alpha \beta }}.\]
Let  \(\mathcal{E} =\mathcal{E} (\boldsymbol{a} )\) be the non-empty set of indices \((j,\alpha ,\beta )\) such that \(a_{j\alpha \beta }=0\); that is,
  \[\mathcal{E} =\mathcal{E} (\boldsymbol{a} )\coloneq \big\{ (j,\alpha ,\beta )\in \{0,1,\dots,p\}\times \{0,1\}\times \{0,1\} \mid  a_{j\alpha \beta }=0\big\}.\]
We will prove that as \(\varepsilon \rightarrow 0^+\),
\begin{equation}
\label{eq:nabla}
\Delta =\left(\sum _{(j,\alpha ,\beta )\in \mathcal{E} } c_{j\alpha \beta }\right)\log(\varepsilon) +O(1).
\end{equation}
Recalling from above that  \(c_{j\alpha \beta }>0\) for all $(j,\alpha,\beta)\in\mathcal{E}$,  (\ref{eq:nabla}) implies that the
directional derivative is negative as $\varepsilon\to 0^+$, and hence $\boldsymbol{a}$ cannot be a local maximum of $\varphi$.

To establish (\ref{eq:nabla}), we apply (\ref{eqn:partial-phi-partial-zjab}) to obtain
  \begin{align}
    \Delta  &= - \sum _{j\in [p]} \sum _{\alpha ,\beta \in \{0,1\}}
    c_{j\alpha \beta }\left[j(-1)^{\alpha +\beta }\log\left(\frac AB\right) + \log\left(\frac{\tau _j}{\tau _0}\right)
      + \log\left(\frac{z_{0\alpha \beta }}{z_{j\alpha \beta }}\right) \right] \notag \\
    &= - \left(\sum _{j\in [p]} \sum _{\alpha ,\beta \in \{0,1\}} jc_{j\alpha \beta }(-1)^{\alpha +\beta }\right)\log\left(\frac AB\right)
    \sum _{j\in [p]}\sum _{\alpha ,\beta \in \{0,1\}} c_{j\alpha \beta }\log\left(\frac{\tau _j}{\tau _0}\right) \notag \\
    &\quad \hspace*{2cm} {} + 
    \sum _\jabss c_{j\alpha \beta }\log(z_{j\alpha \beta })\\
&= \left(\sum _{(j,\alpha ,\beta )\in \mathcal{E} }c_{j\alpha \beta }\right) \log(\varepsilon )+O(1),
    \label{eqn:DELTA}
  \end{align}
  using Lemma~\ref{lemma:ABnonzero} for the final inequality. 
\end{proof}

It remains to consider the cases where \(z=0\) and \(z=1/2\).
In fact, equation (\ref{eqn:phisymm}) implies that these cases are equivalent,
since
\[\varphi \left(0,(z_{j00},z_{j01},z_{j10},z_{j11})_{j\in [p]}\right)
=\varphi \left(1/2,(z_{j01},z_{j00},z_{j11},z_{j00})_{j\in [p]}\right).\]
So 
for the remainder of this appendix, we assume that \(z=0\). 
Then it follows from (\ref{eqn:implicit-z0ab}) that $z_{j\alpha\alpha}=0$ for all $j\in \{0,1,\ldots, p\}$ and $\alpha\in \{0,1\}$.
Hence, it is enough to consider the induced function \(\varphi _0:K_0\rightarrow \mathbb{R} \) on the domain
\begin{align*}
  K_0=\Big\{(z_{j01},z_{j10})_{j\in [p]}\in \mathbb{R} ^{2p} \mid  0&\le z_{j01},z_{j10} \text{ for all } j\in [p],\,\,\,
  0\le \sum _{j\in [p]} z_{j01},\sum _{j\in [p]} z_{j10}\le \dfrac12 \Big\}
\end{align*}
defined by
\[\varphi _0\left((z_{j01},z_{j10})_{j\in [p]}\right)
=\varphi \left(0,(0,z_{j01},z_{j10},0)_{j\in [p]}\right).\]
Defining \(A_0\) and \(B_0\) accordingly,
we obtain
\begin{equation}
  \label{eqn:A0def}
  A_0= p-\sum _{j\in [p]}j(z_{j01}+z_{j10})=\frac d2-B_0,
\end{equation}
and
\begin{align}
  \varphi _0(\boldsymbol{z} )&=A_0\log(A_0)+B_0\log(B_0)
  +\sum _{j=0}^p \, \sum_{\alpha\neq \beta} z_{j\alpha\beta}\, \big(\log(\tau _j)-\log(z_{j\alpha\beta}\big),
  \label{eqn:phi0-def} 
\end{align}
where \(\boldsymbol{z} =(z_{j01},z_{j10})_{j\in [p]}\).

\subsection{The maxima of $\varphi_0$} \label{section:maxima-phi0}

The following lemma gives a complete characterisation of the critical points of \(\varphi _0\) in terms of the function $\theta$ defined in (\ref{eqn:theta-def}).

\begin{lemma}
  \label{lemma:curve}
  Every critical point of \(\varphi_0\) lies on the curve \(\tilde{\boldsymbol{z}} : (0,\infty)\rightarrow K_0\) defined by
    \(x\mapsto \tilde{\boldsymbol{z}} (x)=(\tilde z_{j01},\tilde z_{j10})_{j\in [p]}\) where
  \begin{equation}
    \label{eqn:zjab-irrat}
    \tilde z_{j01}=\tilde z_{j10}
    =\frac{\tau _jx^j}{2\theta (x)}
  \end{equation}
 Moreover,
  \(\tilde{\boldsymbol{z}} (x)\) is a critical point of \(\varphi _0\) if and only if
  \(x>0\) is a root of the polynomial
  \begin{equation}
    \label{eqn:gdef}
    g(x)\coloneq 
    2\big(p\theta (x)-\theta _1(x)\big)(1+x)-d\theta (x).
  \end{equation}
\end{lemma}
\begin{proof}
  Firstly, note that $\tilde{\boldsymbol{z}}$ is a well-defined curve
  which lies in the interior of $K_0$.
 The partial derivatives of \(\varphi _0\) are given by, for all $j\in [p]$ and $\alpha\neq \beta$,
  \[
    \frac{\partial \varphi _0}{\partial z_{j\alpha \beta }}=
    j\left(\log B-\log A\right)
    +\log\left(\frac{\tau _j}{\tau _0}\right)
    +\log(z_{0\alpha \beta })-\log(z_{j\alpha \beta }).
\]
Setting all partial derivatives to zero
implies that,
  for any $j\in [p]$ and $\alpha\neq \beta$, 
  \begin{equation}
    \label{eqn:zjab-boundary}
    z_{j\alpha \beta }=\left(\frac{B_0}{A_0}\right)^j\frac{\tau _j}{\tau _0}z_{0\alpha \beta }.
  \end{equation}
  Summing over all $j\in [p]$ and rearranging, we obtain
  \begin{equation}
    \label{eqn:clever}
    \frac{1-2z_{0\alpha \beta }}{2z_{0\alpha \beta }}=\sum _{j\in [p]}\left(\frac{B_0}{A_0}\right)^j\frac{\tau _j}{\tau _0},
  \end{equation}
  since \(\sum _{j\in [p]}z_{j\alpha \beta }=1/2-z_{0\alpha \beta }\).
  The right hand side of this equation does not depend on \(\alpha \) and \(\beta \), thus,
  \[\frac{1-2z_{001}}{2z_{001}}=\frac{1-2z_{010}}{2z_{010}},\]
  which implies that \(z_{001}=z_{010}\). Together with
  equation (\ref{eqn:zjab-boundary}), this shows that for all $j\in [p]$
  we have \(z_{j01}=z_{j10}\).
  Further, if we define \(x\coloneq B_0/A_0\), then equation (\ref{eqn:clever}) implies that
  \begin{equation}
    \label{eqn:clever2}
    z_{0\alpha \beta }=\frac{1}{2\left(1+\sum _{j\in [p]}\frac{\tau _j}{\tau _0}x^j\right)}
    =\frac{1}{2\sum _{j=0}^p\frac{\tau _j}{\tau _0}x^j}.
  \end{equation}
  Hence we can parameterise the critical points of \(\varphi _0\) in the interior of \(K_0\)
  by substituting these values into equation (\ref{eqn:zjab-boundary}) as follows: for all $j\in [p]$,
  \begin{equation}
    z_{j01}=z_{j10}=\frac{\tau _jx^j}{2\tau _0\sum _{k=0}^p\frac{\tau _k}{\tau _0}x^k}
    =\frac{\tau _jx^j}{2\sum _{k=0}^p\tau _kx^k}
    =\frac{\tau _jx^j}{2\theta (x)}.
  \end{equation}
  By (\ref{eqn:clever2}), this formula also holds when \(j=0\).
  Allowing \(x\) to vary on the interval \((0,\infty )\), we have proved the first statement of the lemma.
  
  Now observe that, by construction,
  the critical points of \(\varphi _0\) are exactly the points such that
  \(x=B_0/A_0\), which are precisely the solutions of the identity
  \[A_0=\frac{d}{2(x+1)}=p-\sum _{j\in [p]} j(z_{j01}+z_{j10})\]
  using (\ref{eqn:A0def}).
  Rearranging, we see that \(x=B_0/A_0\) if and only if \(x\) is a positive root of \(g\),
  completing the proof of the second statement.
\end{proof}

Define \(\tilde \varphi(x)=\varphi_0\big(\boldsymbol{\tilde z}(x)\big)\).
To complete our analysis of the interior of \(K_0\), we must
show that
\(\tilde \varphi(x)<\varphi(\boldsymbol z^\ast)\) for all \(x>0\).
Recall the function $\theta_1$ defined in (\ref{eqn:theta-def}).
If \(\tilde A=A_0(\boldsymbol{\tilde z}(x))=d/2-\tilde B\), then we obtain the following expressions:
\[
\tilde A=p-\frac{\theta_1(x)}{\theta(x)}
\quad
\text{and}
\quad
\tilde \varphi (x)=\tilde A\log\tilde A + \tilde B\log\tilde B
+\log\big(2\theta(x)\big)-\frac{\theta_1(x)}{\theta(x)}\log(x).
\]
Moreover, the derivatives of these expressions are exactly
\[
\tilde A' = - \frac{\theta_2(x)\theta(x)-\theta_1(x)^2}{x\theta(x)^2}
\quad\text{and}\quad
\tilde \varphi'(x)=\tilde A'\left(\log\tilde A-\log\tilde B+\log(x)\right),\]
where \(\theta_2(x)=x\theta_1'(x)\).
Checking that \(\tilde \varphi(x)<\varphi(\boldsymbol z^\ast)\) for any specific \(d\) and \(p\) is a routine exercise:
using standard root counting methods, show that \(g\) has a unique positive root 
and approximate this root numerically. Then, find
an absolute bound on the derivative of \(\tilde\varphi\) and, using the Mean Value Theorem and the
aforementioned approximation,
show that in an appropriate neighbourhood of this root, \(\tilde\varphi\) does not exceed the desired upper bound.
We illustrate this process for the case \((d,p)=(13,3)\), where
\begin{align*}
  g(x)&=
  13442 {{x}^{3}}+60060 {{x}^{2}}-141570 x-240240, \\
  \frac{\theta_1(x)}{\theta(x)}&=\frac{3 x\, \left( {{x}^{2}}+20 x+45\right) }{{{x}^{3}}+30 {{x}^{2}}+135 x+120}, \\
  -\tilde A'&=
  \frac{30 \left( {{x}^{4}}+18 {{x}^{3}}+171 {{x}^{2}}+480 x+540\right) }{{{\left( {{x}^{3}}+30 {{x}^{2}}+135 x+120\right) }^{2}}}.
\end{align*}
The sequence of coefficients of the polynomial \(g\) has exactly one sign change, thus, Descartes' Rule of
Signs implies that \(g\) has a unique positive root, say \(x^\ast\) and this lies on the interval
\([5/2,8/3]\) since
\(g(5/2)=-35035/4<0\) and \(g(8/3)=1734304/27>0\).
Moreover, on this interval, the following inequalities hold:
\[-1< \tilde A' < 0,\quad
1 < \tilde A < 2,\quad
4 < \tilde B < 5.\]
Hence 
on the interval \([5/2,8/3]\) we have
\[-1 < -\log\left(2\right) < \log\tilde A -\log\tilde B+\log(x) < \log\left(\dfrac{4}{3}\right) < 1,\]
which implies that
\(|\tilde \varphi'(x)|< 1.\)
Thus, the Mean Value Theorem implies that
\[\tilde\varphi(x^\ast)<\tilde\varphi(5/2)+\dfrac{8}{3}-\dfrac{5}{2}=\tilde\varphi(5/2)+\dfrac{1}{6}.\]
However,
\begin{align*}
  20.2398
  &\approx\tilde\varphi(5/2)+\dfrac{1}{6} \\
  &=
  \log{\left( \frac{755755}{2}\right) }+\frac{9829 \log{\left( \frac{9829}{2114}\right) }}{2114}-\frac{1215 \log{\left( \frac{5}{2}\right) }}{1057}+\frac{1956 \log{\left( \frac{1956}{1057}\right) }}{1057}+\dfrac{1}{6} \\
  &< 2 \log{(572)}+\frac{13 \log{\left( \frac{13}{4}\right) }}{2}
  = \varphi(\boldsymbol z^\ast)
  \approx 20.3595
\end{align*}
Hence, when \((d,p)=(13,3)\), the maximum of \(\varphi\) does not occur at boundary points corresponding to the critical points of \(\varphi_0\)
in the interior
of \(K_0\). The other cases of Theorem \ref{thm:small-dp} are resolved similarly: see Appendix~\ref{ss:numerical-unique}.
All that remains is to verify that \(\varphi_0\) does not attain its maximum on the boundary of \(K_0\), since this
implies that
\(x^\ast\) above is the maximum of \(\varphi_0\), and hence corresponds
to the maximum of \(\varphi\) on the boundary of \(K\).


\subsection{The Boundary of $K_0$} \label{ss:boundaryko}

Let $\boldsymbol{a}$ be a point on the boundary of $K_0$.
First we suppose that $a_{p01}, a_{p10}$ are not both equal to $1/2$.
As in Section \ref{ss:Kint-directional-stuff}, we approach different boundary points of \(K_0\)
along line segments contained in the interior of \(K_0\). 
Defining \(a_{j\alpha\beta}\) and \(c_{j\alpha\beta}\)
(with \(j\in[p]\) and \(\alpha\ne\beta\))
as in \ref{ss:Kint-directional-stuff},
we approach
a boundary point \(\boldsymbol{a}\) of \(K_0\) along the line segment parameterised by
\(\boldsymbol{z}(\varepsilon)=\boldsymbol{a}+\varepsilon\boldsymbol{c}\).

Let \(\Delta_0\) be \(\Delta\) after substituting
\(z_{j\alpha\beta}=c_{j\alpha\beta}=0\) for all \(j\in[p]\)
and \(\alpha=\beta\in\{0,1\}\). Then
the directional derivative of \(\varphi_0\) is proportional to \(\Delta_0\).
Since at least one of \(z_{p01}\) and \(z_{p10}\) is bounded away from \(1/2\)
as \(\varepsilon\to0^+\), Lemma \ref{lemma:ABnonzero} implies that \(A_0\) and \(B_0\) are bounded away from
\(0\).
Thus, equation (\ref{eqn:DELTA}) implies that as \(\varepsilon\to0^+\),
\[ \Delta_0=\left(\sum_{(j,\alpha,\beta)\in\mathcal{E}} c_{j\alpha\beta}\right)\log(\varepsilon)+O(1)\]
where \(\mathcal{E}\) is defined as in Section \ref{ss:Kint-directional-stuff}.
Since \(\boldsymbol{a}\) is a boundary point, \(\mathcal{E}\ne\emptyset\), thus,
\(\boldsymbol{a}\) is not a local maximum of \(\varphi_0\).

Suppose now that \(a_{p01}=a_{p10}=1/2\),
so \(a_{j01}=a_{j10}=0\) for $j=0,\ldots p-1$
by (\ref{eqn:implicit-z0ab}). Evaluating \(\varphi_0\) at this point $\hat{\boldsymbol{a}}$ gives
\begin{align*}
  \varphi\big(\boldsymbol{z}^{\ast}\big)-\varphi_0\big(\hat{\boldsymbol{a}}\big)
  &=\left(\frac d2\log\left(\frac d4\right)+2\log\left[2\binom dp\right]\right)
  - \left(\frac d2\log\left(\frac d2\right) + \log\left[2\binom dp\right] \right) \\
  &=\log\left[\binom dp 2^{1-d/2}\right]>0.
\end{align*}
Therefore that \(\varphi \big(\boldsymbol{z}^{\ast}\big)>\varphi _0\big(\hat{\boldsymbol{a}}\big)\), using the fact that
 \(\binom{d}{p}2^{1-d/2} > 1\) for all $(d,p)$ in \eqref{eq:list}.

This proves that the maximum of \(\varphi_0\) does not occur on the boundary of \(K_0\),
which concludes our boundary analysis. In particular, we have proved that the maximum of
\(\varphi\) does not occur on the boundary of \(K\), but at the point \(\boldsymbol{z}^\ast\) from Lemma~\ref{eqn:hard-phi-maximum-lemma},
as required.



\subsection{The omitted calculations for $(d,p)\in \{ (6,1), (10,2), (14,3), (17,4)\}$}\label{ss:numerical-unique}

For completeness we provide all details needed to complete the second moment calculations for the values of $(d,p)$ in (\ref{eq:list}).
Firstly, we establish the condition of Lemma~\ref{lemma:polynomial-criterion}, namely that $f$ has no root in $(1,\infty)$.
Let \(x=w-1\). We have the following Taylor expansions:
{\small
\begin{align*}
  (6,1)&: f(w)/48={{x}^{6}}+6 {{x}^{5}}+15 {{x}^{4}}+20 {{x}^{3}}+14 {{x}^{2}}+4 x, \\
  (10,2)&: f(w)/180=
  11 {{x}^{10}}+110 {{x}^{9}}+487 {{x}^{8}}+1256 {{x}^{7}}+2086 {{x}^{6}}+2324 {{x}^{5}} \\
  &\,\quad\quad\quad\quad\quad+1758 {{x}^{4}}+904 {{x}^{3}}+308 {{x}^{2}}+56 x, \\
  (14,3)&: f(w)/2912=13 {{x}^{14}}+182 {{x}^{13}}+1249 {{x}^{12}}+5524 {{x}^{11}}+17204 {{x}^{10}}+38896 {{x}^{9}}+64284 {{x}^{8}}\\
  &\quad\quad\quad\quad\quad\quad+77088 {{x}^{7}}+65538 {{x}^{6}}+37708 {{x}^{5}}+13442 {{x}^{4}}+2488 {{x}^{3}}+
  180 {{x}^{2}}+24 x, \\
  (17,4)&:f(w)/14280=
  29 {{x}^{17}}+493 {{x}^{16}}+4308 {{x}^{15}}+25180 {{x}^{14}}+107500 {{x}^{13}}+348452 {{x}^{12}} \\
  &\quad\quad\quad\quad\quad\quad\quad+873756 {{x}^{11}}+1708421 {{x}^{10}}+2606175 {{x}^{9}}+3083938 {{x}^{8}}+
  2795936 {{x}^{7}} \\
  &\quad\quad\quad\quad\quad\quad\quad+1903902 {{x}^{6}}+946946 {{x}^{5}}+332816 {{x}^{4}}+80872 {{x}^{3}}+13860 {{x}^{2}}+1540 x.
\end{align*}
}
Since all coefficients are nonnegative, in each case, \(f(w)\) has a unique root on \((1,\infty)\).

\noindent Next, using the notation from
Section~\ref{section:maxima-phi0},
we prove that no root of \(g\) gives a global
maximum of \(\varphi\)
according to the parametrisation 
(\ref{eqn:zjab-irrat}).

\noindent{\bf Case $\boldsymbol{(d,p)=(6,1)}$}.\ Here \(g(x)=24x-120\) and 
  \[5.9269\approx\frac{\log{(140625)}}{2}=\tilde\varphi(4)<\varphi(\boldsymbol{z}^\ast)
  = 2 \log{(12)}+3 \log{\left( \dfrac{3}{2}\right) } \approx 6.1862.\]

\noindent{\bf Case $\boldsymbol{(d,p)=(10,2)}$}.\ Here \(g(x)=990 {{x}^{2}}-720 x-7560\), which has positive root
\[ x^\ast\coloneq\frac{1}{11}\left(2\sqrt{235}+4\right)\approx3.1509\in[3,16/5].\]
  Further, we have
  \[\frac{\theta_1(x)}{\theta(x)}=\frac{2 x( x+8) }{(x+2)( x+14) },\quad
  -\tilde A=\frac{16( {{x}^{2}}+7 x+28) }{( x+2)^2( x+14)^2}.\]
  On the interval \([3,16/5]\), we have $0<-\tilde A<1<\tilde A<2$ and $3<\tilde B<4$, so 
  \[ -1<\log\left(\dfrac34\right)<\tilde\varphi(x)<\log\left(\dfrac{11}6\right)<1.\]
  Thus, by the Mean Value Theorem, $\tilde\varphi(x^\ast)<\dfrac{1}{5}+\tilde\varphi(3)$,
  but 
  \[
    13.5544 \approx
    \tilde\varphi\left(3\right)+\dfrac{1}{5} 
    \varphi(\boldsymbol{z}^\ast)
    \approx13.5811.
  \]

\noindent{\bf Case $\boldsymbol{(d,p)=(14,3)}$}.\
Here \(g(x)=18928 {{x}^{3}}+96096 {{x}^{2}}-240240 x-480480\). Since the sequence
  of coefficients of \(g\) has exactly one
  sign change, there is a unique positive root \(x^\ast\).
  Moreover, \(x^\ast\in [25/9,26/9]\) since
  \(g(25/9)=-465920/729<0\) and \(g(26/9)=61111232/729>0\).
  Further, we have
  \[\frac{\theta_1(x)}{\theta(x)}=\frac{3 x( {{x}^{2}}+22 x+55) }{{{x}^{3}}+33 {{x}^{2}}+165 x+165},\quad
  -\tilde A=
  \frac{33(x^4 + 20x^3 +210x^2 + 660x + 825) }{(x^3 + 33x^2 + 165x + 165)^2}.\]
  On the interval \([25/9,26/9]\), we have $0<-\tilde A'<1<\tilde A<2$ and $5<\tilde B<6$, so 
\[  -1<\log\left(\dfrac{25}{54}\right)<\tilde\varphi'(x)<\log\left(\dfrac{52}{45}\right)<1.\]
  Therefore, by the Mean Value Theorem, $\tilde\varphi(x^\ast)<\tilde\varphi\left(\dfrac{25}9\right)+\dfrac19$.
  However,
  \[
    21.9072 \approx\tilde\varphi\left(\dfrac{25}{9}\right)+\dfrac{1}{9} \\
    <
    \varphi(\boldsymbol{z}^\ast) \approx 21.9499.
  \]
  
\noindent{\bf Case $\boldsymbol{(d,p)=(17,4)}$}.\  In this case, \(g(x)=207060 {{x}^{4}}+2598960 {{x}^{3}}+1856400 {{x}^{2}}-16336320 x-15315300\), and
  by Descartes' Rule of Signs, it has a unique positive root \(x^\ast\) which lies in the interval
  \([7/3,5/2]\). We have
  \begin{align*}
    \frac{\theta_1(x)}{\theta(x)}&=
  \frac{4 x( x^3 + 39x^2 + 234x+286) }{x^4 + 52x^3 + 468x^2 + 1144x + 715}, \\
  -\tilde A'&=
  \frac{52(x^6 + 36x^5 +666x^4 + 4796x^3 + 16731x^2 + 25740x + 15730) }{( x^4 + 52x^3 + 468x^2 + 1144x + 715)^2}.
  \end{align*}
  On the interval \([7/3,5/2]\), we have $0<-\tilde A'<1$,\, $2<A<3$ and  $5<B<7$. Hence
  \[
  -1<\log\left(\dfrac23\right)<\tilde\varphi(x)'<\log\left(\dfrac32\right)<1.\]
  Therefore, by the Mean Value Theorem, $\tilde\varphi(x^\ast)<\tilde\varphi\left(\dfrac{5}{2}\right)+\dfrac{1}{6}$, so
  \[
    29.1992 \approx\tilde\varphi\left(\dfrac{5}{2}\right)+\dfrac{1}{6} 
    <
    \varphi(\boldsymbol{z}^\ast) \approx 29.2348.
  \]

\end{document}